\documentclass[12pt]{article}

\usepackage{mathtools}
\usepackage{amsmath,amssymb,amsfonts,amsthm,cases,amscd,amsbsy}
\usepackage{graphicx,epsfig,extarrows}
\usepackage{url}
\usepackage[colorlinks=true,linkcolor=blue,citecolor=blue,urlcolor=blue]{hyperref}
\usepackage{caption,subcaption}
\usepackage{tikz}
\usepackage{tikz-cd}
\usepackage{diagbox}
\usepackage{indentfirst}
\usepackage{arydshln}
\usepackage{enumerate}
\usepackage{setspace}
\usepackage{pifont}
\usepackage{latexsym}
\usepackage{tabularx}
\usepackage{multirow}
\usepackage[normalem]{ulem}
\input xy
\usetikzlibrary{decorations.pathreplacing,decorations.markings}
\usetikzlibrary{cd}
\usetikzlibrary{shapes.geometric}
\usetikzlibrary{arrows,arrows.meta}
\usetikzlibrary{graphs,positioning}
\usetikzlibrary{matrix,decorations.pathmorphing}
\usetikzlibrary{math,calc,intersections,through,angles,arrows.meta,shapes.geometric,shadows,quotes,spy,datavisualization,datavisualization.formats.functions,plotmarks}
\tikzset{every picture/.style={samples=300,smooth,line join=round,thick,>=stealth}}

\numberwithin{equation}{section}

\theoremstyle{plain}
\newtheorem{theorem}{Theorem}[section]

\newtheorem{corollary}[theorem]{Corollary}

\newtheorem{lemma}[theorem]{Lemma}
\newtheorem{proposition}[theorem]{Proposition}

\theoremstyle{definition}
\newtheorem{definition}[theorem]{Definition}

\newtheorem{remark}[theorem]{Remark}

\newtheorem{assumption}[theorem]{Assumption}
\newtheorem{notation}[theorem]{Notation}

\def\XXint#1#2#3{{\setbox0=\hbox{$#1{#2#3}{\int}$}
		\vcenter{\hbox{$#2#3$}}\kern-.5\wd0}}

\DeclareMathSymbol{\subseteq}{\mathrel}{symbols}{"12}
\DeclareMathSymbol{\supseteq}{\mathrel}{symbols}{"13}
\DeclareMathSymbol{\subsetneq}{\mathrel}{AMSb}{"28}
\DeclareMathSymbol{\supsetneq}{\mathrel}{AMSb}{"29}
\DeclareMathSymbol{\nsubseteq}{\mathrel}{AMSb}{"2A}
\DeclareMathSymbol{\nsupseteq}{\mathrel}{AMSb}{"2B}
\DeclareMathDelimiter{\langle}{\mathop}{symbols}{"68}{largesymbols}{"0A}
\DeclareMathDelimiter{\rangle}{\mathclose}{symbols}{"69}{largesymbols}{"0B}

\def\pt{\partial}

\def\ra{\rightarrow}

\def\s{\subseteq}

\def\e{\epsilon}

\def\ol{\overline}

\def\vp{\varphi}
\def\lg{\langle}

\def\rg{\rangle}

\def\es{\emptyset}

\def\Om{\Omega}
\def\la{\lambda}
\def\al{\alpha}

\def\de{\delta}
\def\ga{\gamma}

\def\Ga{\Gamma}

\def\ts{\times}

\def\iy{\infty}

\def\f{\frac}

\def\Lra{\Leftrightarrow}

\def\Span{{\rm{span}}}

\def\df{\mathrm d}

\def\wh{\widehat}

\def\esssup{\operatorname*{ess\ \! sup}}

\def\hra{\hookrightarrow}

\def\mcA{\mathcal{A}}

\newcommand{\R}{\mathbb R}
\newcommand{\N}{\mathbb N}

\DeclareMathOperator{\Div}{div}

\DeclareMathOperator{\dist}{dist}
\DeclareMathOperator{\supp}{supp}

\title{Shape Design for a Class of Degenerate Parabolic Equations with Boundary Point Degeneracy and Its Application to Boundary Observability}

\author{
Dong-Hui Yang\thanks{School of Mathematics and Statistics, Central South University, Changsha 410083, China. Email: donghyang@outlook.com}
\and
Jie Zhong\thanks{Corresponding author. Department of Mathematics, California State University Los Angeles, Los Angeles, CA 90032, USA. Email: jiezhongmath@gmail.com}
}

\date{}

\begin{document}

\maketitle

\begin{abstract}
We study a class of degenerate parabolic equations with boundary point degeneracy in dimensions $N\geq 2$ and investigate the associated boundary observability problem by means of shape design. While one-dimensional degenerate models have been treated in the literature, the genuinely higher-dimensional case remains much more delicate because the degeneracy occurs at a boundary point and the boundary normal trace cannot be extracted directly near the singularity.

We approximate the degenerate equation by a family of uniformly parabolic problems on truncated domains obtained by removing a small neighborhood of the degenerate point. Under a geometric condition on the boundary, we establish uniform estimates for the approximate problems, prove convergence to the solution of the original degenerate equation, and identify the convergence of the boundary normal derivatives under additional regularity.

We then combine this approximation scheme with a parabolic Carleman estimate for the approximate backward equations and derive a boundary observability inequality for the limiting degenerate equation. In this way, we obtain a higher-dimensional parabolic counterpart of the shape-design program previously developed for degenerate hyperbolic equations.
\end{abstract}



\section{Introduction}\label{S1}
	
	Controllability of partial differential equations has been studied extensively in the literature \cite{FG,Fursikov,Lions,Rousseau}. The Hilbert Uniqueness Method (HUM), introduced by J.-L. Lions \cite{Lions}, shows that controllability is equivalent to observability for the corresponding adjoint system.
	
		Degenerate parabolic equations have also been widely studied from the viewpoint of controllability \cite{Alabau,Araruna,Buffe,Cannarsa1,Guo,Wu,Yang}. In the present paper, we consider the genuinely higher-dimensional case $N\geq 2$ in which the degeneracy occurs at a boundary point. This setting is much less understood than the one-dimensional case and requires a different analytical framework near the degenerate boundary.
	
		For both degenerate and non-degenerate parabolic equations, the classical approach to controllability is based on Carleman estimates \cite{Alabau,Araruna,Cannarsa1,FG,Fursikov,Guo,Wu,Yang}. Another well-known approach is the Lebeau--Robbiano strategy \cite{Buffe,Rousseau}, which also depends on suitable Carleman estimates. In higher-dimensional degenerate settings, however, it is often difficult to establish such estimates directly because the regularity near the degenerate boundary may be insufficient. In particular, integration by parts on the degenerate boundary may fail. For this reason, we employ the shape design method introduced in \cite{Guo}. The present work should be viewed as the parabolic counterpart of the shape-design line developed in our earlier study of degenerate hyperbolic equations \cite{YangHyper}. The point is not a formal transposition: in the present setting the relevant adjoint problem is backward parabolic, the key estimate is a parabolic Carleman inequality, and the limit passage must be carried out within the parabolic regularity framework.
	
		Most existing controllability results for degenerate parabolic equations concern operators for which only part of the principal coefficients degenerates. A typical example is $A=\mathrm{diag}(1,\dots,1,x_N^\alpha)$; see \cite{Cannarsa1}. There are also works on more general anisotropic degenerate operators such as $A=\mathrm{diag}(x_1^{\alpha_1},x_2^{\alpha_2})$, or related forms \cite{Araruna,Yang}. By contrast, the controllability problem for the fully isotropic principal operator $A=x_N^\alpha I_N$ (where $I_N=\mathrm{diag}(1,\dots,1)$ is the identity matrix) remains open.
	
	In this paper, we study a boundary observability problem for the degenerate parabolic equation \eqref{12.14.1}, whose degeneracy occurs at a boundary point. Although this setting may look simpler than the case $A=x_N^\alpha I_N$, the operator $A=|x|^\alpha I_N$ is completely degenerate at the point $0\in\mathbb{R}^N$. To handle this difficulty, we impose a geometric assumption on the domain; see Assumption \ref{Assumption (H)}. This assumption arises naturally from the Carleman-weight construction, and it is not clear whether it is optimal. Its role is crucial: it yields the favorable sign needed in the boundary term of the Carleman estimate and, consequently, allows the observability constants for the approximate problems to be chosen uniformly with respect to the approximation parameter.
	 
	We consider the following degenerate parabolic equation:
	\begin{equation}\label{12.14.1}
		\begin{cases}
			\partial_{t}y-\Div(|x|^\alpha \nabla y)=f, &\text{in }Q, \\
			y=0, &\text{on }  \pt Q, \\
			y(0)=y^0,   &\text{in }\Omega,
		\end{cases}
	\end{equation}
		where $\alpha\in (0,1)$ is a given constant, $\Omega\subset\mathbb{R}^N$ with $N\geq 2$ is a bounded domain such that $0\in\pt\Om$ and $\partial\Omega\in C^3$, $Q=\Omega\times (0,T)$ for some $T>0$, $\pt Q=\pt\Om \ts (0,T)$, $y^0\in L^2(\Om)$ is the initial datum, and $f\in L^2(Q)$.

\begin{assumption}\label{Assumption (H)}
	Let $\N=\{1,2,\cdots\}$. Denote $\Om\s B(0,M-1)$, and 
	\begin{equation*}
		w=|x|^\al, \quad \mcA y=-\Div(w\nabla y).
	\end{equation*}
		Thus equation \eqref{12.14.1} can be written as $\pt_t y+\mcA y=f$.
	Here and in what follows, we denote $B(x_0,r)=\{x\in\R^N\colon |x-x_0|<r\}$ for some $r>0$.   
	
	We assume that there exists $R_0>0$ such that 
	\begin{equation*}
		x\cdot \nu(x)\leq 0 \mbox{ for all }x\in \pt\Om\cap B(0,R_0),
	\end{equation*}
	where $\nu$ is the outer normal vector on $\pt\Om$. 
	Hence, for each $\de\in (0,\f{1}{8}R_0)$, there exists a bounded domain $\Om_\de$ of class $C^3$ such that $\Om_\de\s \Om$, $B(0,\de)\cap \Om_\de=\es$, and $\Om-\Om_\de\s B(0,2\de)$, and 
	\begin{equation*}
		x\cdot \nu(x)\leq 0 \mbox{ for all } x\in \pt\Om_\de \cap B(0,R_0),
	\end{equation*}  
	where $\nu$ is the outer normal vector on $\pt\Om_\de$. 
	Indeed, if we take $\wh\Om_\de=\Om-B(0,\de)$, then $x\cdot\nu(x)\leq 0$ for all $x\in \pt\wh \Om_\de\cap B(0,R_0)$, since $x\cdot\nu=-|x|$ on $\pt B(0,\de)\cap \Om$. The domain $\Om_\de$ can then be obtained by mollifying the boundary $\pt \wh\Om_\de \cap B(0,R_0)$. 
	
	Denote 
	\begin{equation}\label{01.12.1}
		\Ga^+=\{x\in\pt\Om\colon x\cdot \nu(x)>0\},
	\end{equation}
	where $\nu$ is the outer normal vector on $\pt \Om$. Then $\Ga^+\s \pt\Om-B(0,R_0)$.
\end{assumption} 

\begin{remark}\label{01.11.R2}
	The domain $\Om\s\R^2$ with boundary portion 
	\begin{equation*}
		\Om \cap B(0,1)=\Om\cap \left\{x\in B(0,1)\colon x_2>|x_1|^{16}\sin\f{1}{|x_1|}\right\},  
	\end{equation*}
	does not satisfy Assumption \ref{Assumption (H)}. 
\end{remark}

The main results of this paper are Theorems \ref{01.11.T3} and \ref{01.12.T2}. Our contribution is twofold. First, in Theorem \ref{01.11.T3}, we establish a shape-design approximation theorem for \eqref{12.14.1} by a sequence of uniformly parabolic equations and prove the convergence properties needed to recover the boundary observation in the higher-dimensional degenerate limit. Second, in Theorem \ref{01.12.T2}, we derive a boundary observability inequality for the backward equation associated with \eqref{12.14.1}. Taken together, these results provide a genuinely higher-dimensional parabolic extension of the shape-design framework.

The paper is organized as follows. In Section \ref{S2}, we collect the preliminary material for \eqref{12.14.1}. In Section \ref{S3}, we introduce the shape-design approximation and derive uniform estimates for the approximate uniformly parabolic equations. Section \ref{S4} is devoted to boundary observability for the approximate problems and to the passage to the limit, which yields the observability inequality for \eqref{12.14.1} stated in Theorem \ref{01.12.T2}.
	
	\section{Preliminaries}\label{S2}
	
	In this section, we collect several preliminary results for equation \eqref{12.14.1}. These include the weighted Sobolev spaces used in the paper, Hardy's inequality, the spectral properties of the operator $\mcA$, and the existence of weak solutions to \eqref{12.14.1}.

	\subsection{Solution Spaces}

	We define
	\begin{equation*}
		H^{1}(\Omega;w) = \left\{u \in L^2(\Omega) \colon \int_\Om w\nabla u\cdot \nabla u\df x<+\iy\right\}, 
	\end{equation*}
	where $\nabla u=(\f{\pt u}{\pt x_1},\cdots, \f{\pt u}{\pt x_N})$. The inner product and norm on $H^1(\Omega; w)$ are defined by
	\begin{equation*}
		(u, v)_{H^1(\Omega; w)} =\int_\Om uv\df x+\int_\Om (\nabla u\cdot \nabla v)w\df x, \quad \|u\|_{H^1(\Om;w)}=(u,u)_{H^1(\Om;w)}^\f{1}{2}, 
	\end{equation*} 
	respectively. 
	We also define
	\begin{equation*}
		H_{0}^1(\Omega; w) = \text{the closure of } \mathcal{D}(\Omega)\equiv C_0^\iy(\Om) \text{ in } H^1(\Omega; w). 
	\end{equation*} 
	We denote by $H^{-1}(\Omega; w)$ the dual space of $H_{0}^1(\Omega; w)$ with pivot space $L^2(\Om)$, which is a subspace of $\mathcal{D}'(\Omega)$.  
	
	Finally, we define
	\begin{equation*}
		H^2(\Om;w)=\left\{u\in H^1(\Om;w)\colon \int_\Om (\mcA u)^2\df x<+\iy\right\}.
	\end{equation*}
	Its inner product and norm are defined by
	\begin{equation*}
		(u,v)_{H^2(\Om;w)}=(u,v)_{H^1(\Om;w)}+\int_\Om (\mcA u)(\mcA v)\df x, \quad \|u\|_{H^2(\Om;w)}=(u,u)_{H^2(\Om;w)}^\f{1}{2}, 
	\end{equation*} 
	respectively.

	It is well known that the spaces
	\begin{equation*}
		H_0^1(\Om;w),\quad H^1(\Om;w), \mbox{ and } H^2(\Om;w)
	\end{equation*}
	are Hilbert spaces
	(see \cite{GC,Heinonen}).
	
	We denote 
	\begin{equation*}
		D(\mcA)=H^2(\Om;w)\cap H_0^1(\Om;w). 
	\end{equation*}
	
	\begin{lemma}\label{08.15.L1}
		Let $N \geq 2$ and $\alpha \in (0, 1)$. Then, for all $u \in H_0^1(\Omega; w)$, the following inequality holds:
		\begin{equation*}
			\int_\Om |x|^{\al-2}u^2\df x\leq \f{4}{(N-2+\al)^2}\int_\Om w\nabla u\cdot \nabla u\df x. 
		\end{equation*} 
	\end{lemma}
	
	\begin{proof}
		Set $\Om_\e=\Om-\ol{B(0,\e)}$ with $\e>0$ sufficiently small. Let $u\in H_0^1(\Om;w)$. Note that $u\in H^1(\Om_\e)$ and
		\begin{equation*}
			\begin{split}
				2\int_{\Om_\e}|x|^{\al-2}z(x\cdot \nabla z)\df x
				&=\int_{\Om_\e} |x|^{\al-2}x\cdot \nabla z^2\df x\\
				&=\int_{\Om_\e} \Div(z^2|x|^{\al-2}x)\df x-\int_{\Om_\e}z^2\Div(|x|^{\al-2}x)\df x\\
				&=\int_{\Om\cap \pt B(0,\e)}z^2|x|^{\al-2}(x\cdot\nu)\df S-(N+\al-2)\int_{\Om_\e}z^2|x|^{\al-2}\df x, 
			\end{split}
		\end{equation*}
		then,  from $x\cdot\nu<0$ on $\Om\cap \pt B(0,\e)$, we get 
		\begin{equation*}
			\begin{split}
				(N+\al-2)\int_{\Om_\e}|x|^{\al-2}z^2\df x
				&\leq  (N+\al-2)\int_{\Om_\e}|x|^{\al-2}z^2\df x-\int_{\pt\Om_\e}|x|^{\al-2}z^2(x\cdot\nu)\df S\\
				&=-2\int_{\Om_\e}|x|^{\al-2}z(x\cdot\nabla z)\df x\leq 2\int_{\Om_\e}(|x|^\f{\al-2}{2}|z|)(|x|^\f{\al}{2}|\nabla z|)\df x\\
				&\leq 2\left(\int_{\Om_\e}|x|^{\al-2}z^2\df x\right)^\f{1}{2}\left(\int_{\Om_\e}|x|^\al |\nabla z|^2\df x\right)^\f{1}{2}, 
			\end{split}
		\end{equation*}
		and hence
		\begin{equation*}
			\begin{split}
				(N+\al-2)\left(\int_{\Om_\e}|x|^{\al-2}z^2\df x\right)^\f{1}{2}\leq 2\left(\int_{\Om_\e}|x|^\al |\nabla z|^2\df x\right)^\f{1}{2}.
			\end{split}
		\end{equation*}
		Letting $\e\ra 0$ and using $\Om_\e\ra \Om$, we obtain the desired estimate.
	\end{proof}
	
	\begin{remark}\label{08.16.R1}
		From Lemma \ref{08.15.L1},  we have
		\begin{equation}\label{08.16.8}
			\int_\Om u^2\df x \leq \f{4M^{2-\al}}{(N-2+\al)^2}\int_\Om w\nabla u\cdot \nabla u\df x,
		\end{equation}
		for all $u\in H_0^1(\Om;w)$. This is  a Poincar\'{e} inequality.
		In particular, the norm
		\begin{equation}\label{08.19.2}
			\|u\|_{H_0^1(\Omega; w)} = \left(\int_\Omega (\nabla u \cdot \nabla u) w \, \mathrm{d}x\right)^{\frac{1}{2}}
		\end{equation}
		is an equivalent norm in $H_0^1(\Omega; w)$. Hereafter, we use \eqref{08.19.2} to define the norm of $H_0^1(\Omega; w)$.
	\end{remark}
	 
	\begin{lemma}\label{08.15.L4}
		The embedding $H_0^1(\Omega; w) \hookrightarrow L^2(\Omega)$ is compact.
	\end{lemma}
	
\begin{proof}
	Let $\{u_n\}_{n\in\N}\s H_0^1(\Om;w)$ be a bounded sequence, i.e. $\|u_n\|_{H_0^1(\Om;w)}\leq C_0$ for all $n\in\N$ with some fixed constant $C_0>0$. Then there exist a subsequence of $\{u_n\}_{n\in\N}$, still denoted by $\{u_n\}_{n\in\N}$, and a function $u_0\in H_0^1(\Om;w)$ such that
	\begin{equation*}
		u_n\ra u_0 \mbox{ weakly in } H_0^1(\Om;w). 
	\end{equation*}
	We now prove that $u_n\ra u_0$ in $L^2(\Om)$. We may assume $u_0=0$ by replacing $u_n$ with $u_n-u_0$.
	
	Let $\de>0$. On the one hand, for $\e>0$ small enough, Lemma \ref{08.15.L1} gives
	\begin{equation*}
		\begin{split} 
			\int_{B(0,\e)} u_n^2\df x
			&=\int_{B(0,\e)} |x|^{2-\al}|x|^{\al-2}u_n^2\df x\leq \e^{2-\al}C^2\int_\Om |x|^\al |\nabla u_n|^2\df x\leq C^2C_0^2\e^{2-\al}, 
		\end{split} 
	\end{equation*} 
	where the positive constant $C$ is the constant in Lemma \ref{08.15.L1}. 
	Choose $\e_0>0$ such that for all $\e\in (0,\e_0]$ and all $n\in\N$ we have $\int_{\Om-\Om_\e}u_n^2\df x<\f{1}{4}\de^2$. On the other hand, we consider the sequence $\{u_n|_{\Om_{\e_0}}\}_{n\in\N}$. Since $|x|\geq \e_0$ for all $x\in \Om_{\e_0}$, it follows from
	\begin{equation*}
		\e_0^\al \int_{\Om_{\e_0}} |\nabla u|^2\df x\leq \int_{\Om_{\e_0}} |\nabla u|^2|x|^\al \df x\leq C_0,
	\end{equation*}
	we get $\{u_n|_{\Om_{\e_0}}\}_{n\in\N}\s H^1(\Om_{\e_0})$ is a bounded sequence. By the classical Sobolev compact embedding theorem, there exists a subsequence $\{u_{n_k}|_{\Om_{\e_0}}\}_{k\in\N}$ of  $\{u_n|_{\Om_{\e_0}}\}_{n\in\N}$  such that $u_{n_k}|_{\Om_{\e_0}}\ra 0$ strongly in $L^2(\Om_{\e_0})$.  Moreover, there exists $k_0\in\N$ such that for all $k\geq k_0$ we have 
	\begin{equation*}
		\left\|u_{n_k}|_{\Om_{\e_0}}\right\|_{L^2(\Om_{\e_0})}^2<\f{1}{4}\de^2. 
	\end{equation*}
	Combining the two estimates, we obtain the desired convergence in $L^2(\Om)$.
\end{proof}
	
\subsection{Spectrum}
	
From Remark \ref{08.16.R1} and Lemma \ref{08.15.L4}, we deduce that the degenerate partial differential operator $\mcA$ has discrete point spectrum
\begin{equation}\label{11.14.3}
	0<\la_1< \la_2\leq \la_3\leq  \cdots \ra +\iy, 
\end{equation}
that is, $\la_n\ (n\in\N)$ satisfies
\begin{equation}\label{12.12.8}
	\begin{cases}
		\mcA \Phi_n=\la_n \Phi_n, & \mbox{in }\Om, \\
		\Phi_n=0, &\mbox{in }\pt\Om. 
	\end{cases}
\end{equation}
Moreover, from \eqref{08.16.8},  we get
\begin{equation}\label{11.30.9}
	\la_1=\inf_{0\neq u\in H_0^1(\Om;w)}\f{\int_\Om w\nabla u\cdot \nabla u\df x}{\int_\Om u^2\df x}\geq \f{(N-2+\al)^2}{4M^{2-\al}}.
\end{equation}

\begin{notation} \label{12.11.N1}
We denote by $\Phi_n(x)$ the $n$th eigenfunction of $\mcA$ corresponding to the eigenvalue $\la_n\ (n\in\N)$. Then $\{\Phi_n\}_{n\in\N}$ is an orthonormal basis of $L^2(\Om)$ and an orthogonal subset of $H_0^1(\Om;w)$. See \cite[Theorem 7 (p.~728) in Appendix D]{Evans} or the proof of Lemma \ref{06.28.L1}.
\end{notation}

\begin{lemma}\label{06.28.L1}
	Let $u=\sum_{i=1}^\iy u_i \Phi_i\in H_0^1(\Om;w)$ with $u_i=(u,\Phi_i)_{L^2(\Om)}$ for all $i\in\N$. We have $\nabla u=\sum_{i=1}^\iy u_i\nabla \Phi_i$ and $\|u\|_{H_0^1(\Om;w)}=(\sum_{i=1}^\iy u_i^2\la_i)^\f{1}{2}$, and
	\begin{equation*}
		u\in H^2(\Om;w)\Lra \sum_{i=1}^\iy u_i^2\la_i^2<\iy,
	\end{equation*}
	and
	\begin{equation*}
		\mcA u=\sum_{i=1}^\iy u_i\la_i\Phi_i \mbox{ and } \|\mcA u\|_{L^2(\Om)}=\left(\sum_{i=1}^\iy u_i^2\la_i^2\right)^\f{1}{2}.
	\end{equation*}
\end{lemma}

\begin{proof}
	From \eqref{11.14.3}, we have
	\begin{equation}\label{06.04.3}
		\int_\Om (\nabla \Phi_k\cdot\nabla \Phi_l)w\df x=\de_{kl}\la_k \mbox{ for } k,l\in \N,
	\end{equation}
	where $\de_{kl}$ is the Kronecker delta function, defined as $\de_{kl}=1$ for $k=l$ and $\de_{kl}=0$ for $k\neq l$.
		We now show that $\{\la_k^{-\f{1}{2}}\Phi_k\}_{k=1}^\iy$ forms an orthonormal basis of $H_0^1(\Om;w)$.
		From \eqref{06.04.3}, it is straightforward to verify that $\{\la_k^{-\f{1}{2}}\Phi_k\}_{k=1}^\iy$ is an orthonormal subset of $H_0^1(\Om;w)$. To prove that it is a basis, suppose by contradiction that there exists $0\neq u\in H_0^1(\Om;w)$ such that
	\begin{equation*}
		\int_\Om (\nabla \Phi_k\cdot \nabla u)w\df x=0 \mbox{ for all }k\in\N.
	\end{equation*}
	Given that $\{\Phi_k\}_{k\in\N}$ is an orthonormal basis of $L^2(\Om)$, for $u\in H_0^1(\Om;w)$, we can express
	\begin{equation}\label{06.04.4}
		u=\sum_{k=1}^\iy d_k\Phi_k \mbox{ where }  d_k=(u, \Phi_k)_{L^2(\Om)}, k\in\N.
	\end{equation}
	Then  we have
	\begin{equation*}
		0=(u,\Phi_k)_{H_0^1(\Om;w)}=\int_\Om (\nabla \Phi_k\cdot \nabla u)w\df x=\la_k\int_\Om \Phi_ku\df x =\la_kd_k,
	\end{equation*}
	which implies $d_k=0$. This leads to $u=0$, a contradiction. 
	
		From the above argument, since $u\in H_0^1(\Om;w)$, we have
	\begin{equation*}
		\nabla u=\sum_{i=1}^\iy e_i\la_i^{-\f{1}{2}}\nabla\Phi_k, \mbox{ and } \|u\|_{H_0^1(\Om;w)}^2=\sum_{i=1}^\iy e_i^2,
	\end{equation*}
	and 
	\begin{equation*}
		e_i=\int_\Om \left(\nabla u\cdot \la_i^{-\f{1}{2}}\nabla\Phi_i\right) w\df x=\la_i^{-\f{1}{2}}\int_\Om (\nabla u\cdot \nabla \Phi_i)w\df x=\la_i^\f{1}{2}\int_\Om u\Phi_i\df x=\la_i^\f{1}{2}u_i, 
	\end{equation*} 
		hence $\nabla u=\sum_{i=1}^\iy u_i\nabla\Phi_i$. Moreover, $\|u\|_{H_0^1(\Om;w)}=(\sum_{i=1}^\iy \la_iu_i^2)^\f{1}{2}$.
	
	Let $u\in H^2(\Om;w)$. Take $\vp_n=\sum_{i=1}^n u_i\Phi_i\in H_0^1(\Om;w)\cap H^2(\Om;w)$ for each $n\in\N$, then from
	\begin{equation*}
		\begin{split}
			(\mcA u, \mcA\vp_n)_{L^2(\Om)}
			&=\sum_{i=1}^n u_i\la_i \int_\Om (\mcA u)\Phi_i\df x=\sum_{i=1}^n u_i\la_i\int_\Om u\mcA\Phi_i\df x\\
			&=\sum_{i=1}^n u_i\la_i^2\int_\Om u\Phi_i\df x=\sum_{i=1}^n u_i^2\la_i^2
		\end{split}
	\end{equation*}
	and $\|\mcA\vp_n\|_{L^2(\Om)}^2=\sum_{i=1}^n u_i^2\la_i^2$ we obtain
	\begin{equation*}
		\sum_{i=1}^nu_i^2\la_i^2\leq \|\mcA u\|_{L^2(\Om)}^2
	\end{equation*}
	for all $n\in\N$ by Cauchy inequality. This implies that $\sum_{i=1}^\iy u_i^2\la_i^2\leq \|\mcA u\|_{L^2(\Om)}^2<\iy$.

	Let $\sum_{i=1}^\iy u_i^2\la_i^2<\iy$. For each  $\vp\in C_0^\iy(\Om)$,  we have
	\begin{equation*}
		\begin{split}
			(\mcA u, \vp)_{L^2(\Om)}
			&=\int_\Om (\nabla u\cdot \nabla\vp)w\df x=\sum_{i=1}^\iy u_i \int_\Om (\nabla\Phi_i\cdot \nabla\vp)w\df x=\sum_{i=1}^\iy u_i\la_i\int_\Om \Phi_i\vp\df x.
		\end{split}
	\end{equation*}
	Note that
	\begin{equation*}
		\int_\Om \left(\sum_{i=1}^n u_i\la_i\Phi_i\right)^2\df x=\sum_{i=1}^n u_i^2\la_i^2\leq \sum_{i=1}^\iy u_i^2\la_i^2<\iy \mbox{ for all } n\in\N,
	\end{equation*}
		that is, $\sum_{i=1}^\iy u_i\la_i\Phi_i\in L^2(\Om)$. Hence
	\begin{equation*}
		(\mcA u, \vp)_{L^2(\Om)}=\left(\sum_{i=1}^\iy u_i\la_i\Phi_i, \vp\right)_{L^2(\Om)}.
	\end{equation*}
	This implies that $\|\mcA u\|_{L^2(\Om)}^2=\sum_{i=1}^\iy u_i^2\la_i^2$ and $\mcA u=\sum_{i=1}^\iy u_i\la_i\Phi_i$.
\end{proof}

\subsection{Weak Solutions}

\begin{definition}\label{12.11.D1}
		We say that
	\begin{equation*}
		y\in L^2(0,T; H_0^1(\Om;w)), \mbox{ with } \pt_ty\in L^2(0,T; H^{-1}(\Om;w))
	\end{equation*}
	is a weak solution of \eqref{12.14.1} with respect to $(y^0,f)$ if
	
		(i) for every $v\in H_0^1(\Om;w)$ and for a.e.~$t\in (0,T)$, we have
	\begin{equation*}
		\lg\pt_{t}y, v\rg_{H^{-1}(\Om;w),H_0^1(\Om;w)}+\int_\Om w\nabla y\cdot \nabla v\df x=\int_\Om fv\df x, 
	\end{equation*}
	
	(ii) $y(0)=y^0$. 
\end{definition}

\begin{theorem}\label{01.11.T1} 
		Let $\al\in (0,1)$. Then equation \eqref{12.14.1} with respect to $(y^0,f)$ admits a unique weak solution
		\begin{equation*}
			y\in L^2(0,T; H_0^1(\Om;w)), \mbox{ with } \pt_ty\in L^2(0,T; H^{-1}(\Om;w))
		\end{equation*} 
		and this solution satisfies the estimate
		\begin{equation}\label{01.11.7}
			\begin{split} 
			&\esssup_{t\in [0,T]}\|y(t)\|_{L^2(\Om)}+\|y\|_{L^2(0,T;H_0^1(\Om;w))}+\|\pt_ty\|_{L^2(0,T; H^{-1}(\Om;w))}\\
			&\leq C\left(\|f\|_{L^2(Q)}+\|y^0\|_{L^2(\Om)}\right), 
			\end{split} 
		\end{equation}
		where the constant $C>0$ depends only on $\al, N, T$, and $M$.
		
		Moreover, if we further assume that $\Om$ satisfies Assumption \ref{Assumption (H)} and that $y^0\in H_0^1(\Om;w)$, then the weak solution belongs to
		\begin{equation*}
			y\in L^2(0,T; D(\mcA))\cap C([0,T]; H_0^1(\Om;w)), \mbox{ with } \pt_ty\in L^2(Q), 
		\end{equation*}
		and satisfies the estimate
		\begin{equation}\label{01.11.8}
			\begin{split} 
			&\esssup_{t\in [0,T]}\|y(t)\|_{H_0^1(\Om;w)}+\|y\|_{L^2(0,T; H^2(\Om-B(0,R_0)))}+\|y\|_{L^2(0,T; D(\mcA))}+\|\pt_ty\|_{L^2(Q)}\\
			&\leq C\left(\|f\|_{L^2(Q)}+\|y^0\|_{H_0^1(\Om;w)}\right), 
			\end{split} 
		\end{equation}
		where the constant $C>0$ depends only on $\al, R_0, N, T, M$, and $\Om-B(0,\f{1}{2}R_0)$.
\end{theorem}

\begin{proof}
		We prove this theorem in several steps.
	
	{\it Step 1}. Galerkin method.
	
		Let $m\in\N$. Define
	\begin{equation*}
		y^m(x,t)=\sum_{n=1}^m y_n^m(t)\Phi_n(x), \quad f^m(x,t)=\sum_{n=1}^m f_n^m(t)\Phi_n(x), 
	\end{equation*}
		where $y_n^m(t)$, for $t\in [0,T]$ and $n=1,\cdots, m$, solves
	\begin{equation}\label{01.11.1}
		\f{\df}{\df t}y_n^m(t)+\la_n y_n^m(t)=f_n^m(t), \mbox{ with } f_n^m(t)=(f,\Phi_n)_{L^2(\Om)}, n\in\N, 
	\end{equation}
	with initial data
	\begin{equation}\label{01.11.2}
		y_n^m(0)=(y^0,\Phi_n)_{L^2(\Om)}\equiv y_n^0, \ n=1,\cdots, m. 
	\end{equation}
		It is clear that
	\begin{equation*}
		y_n^m(t)=y_n^0e^{-\la_nt}+\int_0^te^{\la_n(s-t)}f_n(s)\df s, \ n=1,\cdots, m. 
	\end{equation*}
		Note that \eqref{01.11.1} is equivalent to
	\begin{equation}\label{01.11.3}
		(\pt_ty^m,\Phi_n)_{L^2(\Om)}+(y^m,\Phi_n)_{H_0^1(\Om;w)}=(f^m,\Phi_n)_{L^2(\Om)}, \  n=1,\cdots, m, \  t\in [0,T]. 
	\end{equation}
	
	{\it Step 2}. Energy estimate.
	
		Multiplying \eqref{01.11.3} by $y_n^m(t)$ and summing over $n=1,\cdots, m$, we obtain
	\begin{equation*}
		(\pt_ty^m, y^m)_{L^2(\Om)}+(y^m,y^m)_{H_0^1(\Om;w)}=(f^m,y^m)_{L^2(\Om)} \mbox{ for a.e.}\ t\in [0,T]. 
	\end{equation*}
		Note that $(\pt_ty^m,y^m)_{L^2(\Om)}=\f{1}{2}\f{\df}{\df t}\|y^m\|_{L^2(\Om)}^2$. Using the second inequality in \eqref{08.16.8}, we get
	\begin{equation*}
		\f{\df}{\df t}\|y^m\|_{L^2(\Om)}^2+\|y^m\|_{H_0^1(\Om;w)}^2\leq C\|f\|_{L^2(\Om)}^2,
	\end{equation*}
		where the constant $C>0$ depends only on $\al, N$, and $M$. Integrating over $[0,t]$ for any $t\in (0,T]$, we get
	\begin{equation}\label{01.11.4}
		\begin{split}
			\esssup_{t\in [0,T]}\|y^m(t)\|_{L^2(\Om)}+\|y^m\|_{L^2(0,T; H_0^1(\Om;w))}\leq \|y^0\|_{L^2(\Om)}+C\|f\|_{L^2(Q)}, 
		\end{split}
	\end{equation}
	where the constant $C>0$ depends only on $\al, N$ and $M$. 
	
	For each $v\in H_0^1(\Om;w)$ with $\|v\|_{H_0^1(\Om;w)}\leq 1$, write $v=v^1+v^2$, where $v^1\in \Span \{\Phi_n\}_{n=1}^m$ and $(v^2, \Phi_n)=0$ for $k=1,\cdots, m$. Then $\|v^1\|_{H_0^1(\Om;w)}\leq \|v\|_{H_0^1(\Om;w)}\leq 1$, and from \eqref{01.11.3} we obtain 
	\begin{equation*}
		\lg \pt_ty^m, v\rg_{H^{-1}(\Om;w),H_0^1(\Om;w)}=(\pt_ty^m,v)_{L^2(\Om)}=(\pt_ty^m, v^1)=(f^m,v^1)-(y^m, v^1)_{H_0^1(\Om;w)}. 
	\end{equation*}
		This implies that
	\begin{equation*}
		\|\pt_ty^m\|_{H^{-1}(\Om;w)}\leq C\|f\|_{L^2(\Om)}+\|y^m\|_{H_0^1(\Om;w)}
	\end{equation*}
		by \eqref{08.16.8}. Hence
	\begin{equation}\label{01.11.5}
		\|\pt_ty^m\|_{L^2(0,T; H^{-1}(\Om;w))}\leq \|y^0\|_{L^2(\Om)}+ C\|f\|_{L^2(Q)} 
	\end{equation}
	by \eqref{01.11.4}, 
	where the constant $C>0$ depends only on $\al, N$ and $M$. 
	
	{\it Step 3}. Approximation. 
	
		From \eqref{01.11.4} and \eqref{01.11.5}, there exists a function
	\begin{equation*}
		z\in L^2(0,T; H_0^1(\Om;w)), \mbox{ with } \pt_tz\in L^2(0,T; H^{-1}(\Om;w))
		\end{equation*}
		such that
	\begin{equation}\label{01.11.6}
		\begin{split}
			y^m
			&\ra z\mbox{ weak star in } L^\iy(0,T; L^2(\Om)),\\
			y^m
			&\ra z\mbox{ weakly in }L^2(0,T; H_0^1(\Om;w)), \\
			\pt_ty^m
			&\ra \pt_t z\mbox{ weakly in }L^2(0,T;H^{-1}(\Om;w)). 
		\end{split}
	\end{equation}
	Moreover, we have \eqref{01.11.7}, and 
	\begin{equation*}
		z\in C([0,T]; L^2(\Om)). 
	\end{equation*}
	
	{\it Step 4}. We show that $z$ is a solution of \eqref{12.14.1}. 
	
		Fix $k\in\N$. Choose a function
		\begin{equation*}
			\psi(t)=\sum_{n=1}^k \psi_n(t)\Phi_n,
		\end{equation*}
		where $\psi_n\in C^\iy([0,T])$ for $n=1,\cdots,k$. Then, for all $m\geq k$, multiplying \eqref{01.11.3} by $\psi_n(t)$, summing over $n=1,\cdots, k$, and integrating over $[0,T]$, we obtain
	\begin{equation}\label{01.11.11}
		\begin{split}
			\int_0^T \lg \pt_t y^m, \psi\rg_{H^{-1}(\Om;w),H_0^1(\Om;w)}\df t+\iint_Q w\nabla y^m \cdot\nabla \psi\df x\df t=\iint_Q f^m\psi\df x\df t. 
		\end{split}
	\end{equation}
			Passing to a subsequence if necessary and letting $m\ra\iy$, we obtain
	\begin{equation}\label{01.11.9}
		\int_0^T\lg \pt_t z,\psi\rg_{H^{-1}(\Om;w),H_0^1(\Om;w)}\df t+\iint_Q w\nabla z\cdot\nabla\psi\df x\df t=\iint_Q f\psi\df x\df t
	\end{equation}
			by \eqref{01.11.6}. Since this class of test functions is dense in $L^2(0,T; H_0^1(\Om;w))$, it follows that
		\begin{equation*}
			\lg \pt_tz, v\rg_{H^{-1}(\Om;w),H_0^1(\Om;w)}+(z, v)_{H_0^1(\Om;w)}=(f,v)_{L^2(\Om)} 
		\end{equation*}
		for all $v\in H_0^1(\Om;w)$ and for a.e.~$t\in [0,T]$.
		
		We next show that $z(0)=y^0$.
		
			From \eqref{01.11.9} and the fact that $z\in C([0,T]; L^2(\Om))$, we obtain
	\begin{equation}\label{01.11.10}
		-\iint_Q z\pt_t\psi\df x\df t+\iint_{Q}w\nabla z\cdot \nabla\psi\df x\df t=\iint_Q f\psi\df x\df t+(z(0),\psi(0))_{L^2(\Om)}
	\end{equation}
		for each $\psi\in C^1([0,T]; H_0^1(\Om;w))$ with $\psi(T)=0$. On the other hand, \eqref{01.11.11} yields
	\begin{equation*}
		\begin{split}
			-\iint_Q y^m\pt_t\psi\df x\df t+\iint_Qw\nabla y^m\cdot \nabla\psi\df x\df t=\iint_Qf^m\psi\df x\df t+(y^m(0),\psi(0)).
		\end{split}
	\end{equation*}
			Passing again to a subsequence if necessary and letting $m\ra \iy$, it follows from \eqref{01.11.6} and \eqref{01.11.2} that
	\begin{equation*}
		-\iint_Q z\pt_t\psi\df x\df t+\iint_Q w\nabla z\cdot\nabla \psi\df x\df t=\iint_Q f\psi\df x\df t+(y^0,\psi(0))_{L^2(\Om)}. 
	\end{equation*}
			Comparing this identity with \eqref{01.11.10}, we conclude that $z(0)=y^0$.
	
	{\it Step 5}. Uniqueness. 
	
			Let $\xi=y-z$. Then $\xi$ solves \eqref{12.14.1} with respect to $(0,0)$. Taking $v=\xi$ in Definition \ref{12.11.D1} (i), we obtain
	\begin{equation*}
		\f{1}{2}\f{\df}{\df t}\|\xi\|_{L^2(\Om)}^2+\|\xi\|_{H_0^1(\Om;w)}^2=0. 
	\end{equation*} 
			Integrating over $(0,t)$ for each $t\in [0,T]$, we conclude that $\xi=0$, and hence the weak solution is unique.
	
	{\it Step 6}. Improved regularity. 
	
	Multiplying \eqref{01.11.3} by $\f{\df}{\df t}y_n^m$, summing $n=1,\cdots, m$, we get
	\begin{equation*}
		(\pt_ty^m, \pt_ty^m)_{L^2(\Om)}+(y^m, \pt_ty^m)_{H_0^1(\Om;w)}=(f^m,\pt_ty^m)_{L^2(\Om)} \mbox{ for a.e.}\ t\in [0,T], 
	\end{equation*}
		Note that $(y^m, \pt_ty^m)_{H_0^1(\Om;w)}=\f{1}{2}\f{\df }{\df t}\|y^m\|_{H_0^1(\Om;w)}^2$. Therefore, we obtain
	\begin{equation*}
		\begin{split}
			\|\pt_ty^m\|_{L^2(\Om)}^2+\f{\df}{\df t}\|y^m\|_{H_0^1(\Om;w)}^2\leq \|f\|_{L^2(\Om)}^2 \mbox{ for a.e.}\ t\in [0,T]. 
		\end{split}
	\end{equation*}
	Integrating on $(0,t)$ for each $t\in [0,T]$, we get
	\begin{equation}\label{03.08.1}
		\begin{split} 
		&\esssup_{t\in [0,T]}\|y^m(t)\|_{H_0^1(\Om;w)}^2+\|\pt_ty^m\|_{L^2(Q)}^2\\
		&\leq 2\|f\|_{L^2(Q)}^2+2\|y^m(0)\|_{H_0^1(\Om;w)}^2\\
		&=2\|f\|_{L^2(Q)}^2+2\sum_{n=1}^m (y_n^0)^2\la_n=2\|f\|_{L^2(Q)}^2+2\|y^0\|_{H_0^1(\Om;w)}^2
		\end{split} 
	\end{equation}
		by Lemma \ref{06.28.L1}. This implies that
	\begin{equation}\label{01.11.13}
		\begin{split}
			\esssup_{t\in [0,T]}\|y(t)\|_{H_0^1(\Om;w)}^2+\|\pt_ty\|_{L^2(Q)}^2\leq 2\|f\|_{L^2(Q)}^2+2\|y^0\|_{H_0^1(\Om;w)}^2
		\end{split}
	\end{equation}
	
		Multiplying \eqref{01.11.3} by $\la_n$ and summing over $n=1,\cdots, m$, we get
	\begin{equation*}
		(\pt_ty^m, \mcA y^m)_{L^2(\Om)}+(y^m, \mcA y^m)_{H_0^1(\Om;w)}=(f^m, \mcA y^m)_{L^2(\Om)}. 
	\end{equation*}
		that is,
	\begin{equation*}
		\|\mcA y^m\|_{L^2(\Om)}^2=(f^m,\mcA y^m)_{L^2(\Om)}-(\pt_ty^m, \mcA y^m)_{L^2(\Om)}
	\end{equation*}
		because $\mcA y^m\in H_0^1(\Om;w)$. This implies that
		\begin{equation*}
			\|\mcA y^m\|_{L^2(\Om)}\leq \|f^m\|_{L^2(\Om)}+\|\pt_ty^m\|_{L^2(\Om)}, 
		\end{equation*}
			Together with \eqref{03.08.1}, this yields
		\begin{equation*}
			\|\mcA y^m\|_{L^2(Q)}\leq C\left(\|f\|_{L^2(Q)}+\|y^0\|_{H_0^1(\Om;w)}\right). 
		\end{equation*}
			Combining this with \eqref{01.11.6} and Step 5, we get
		\begin{equation}\label{03.08.3}
			\|\mcA y\|_{L^2(Q)} \leq C\left(\|f\|_{L^2(Q)}+\|y^0\|_{H_0^1(\Om;w)}\right). 
		\end{equation}
			Together with \eqref{01.11.13}, this gives \eqref{01.11.8}, where the constant $C>0$ depends only on $\al, N, T$, and $M$.
	
	Choosing $\zeta\in C^\iy(\ol\R^N), 0\leq \zeta\leq 1$ such that 
	\begin{equation*}
		\zeta=1 \mbox{ on }\R^N-B(0,R_0), \quad \zeta=0 \mbox{ on } B\left(0,\f{3}{4}R_0\right),   \mbox{ and } |D\zeta|\leq CR_0^{-1}, |D^2\zeta|\leq CR_0^{-2}, 
	\end{equation*}
	where the positive constants $C$ are absolute. Note that $\xi=\zeta y^m$ is a solution of the following uniformly elliptic equation
	\begin{equation*}
		\begin{cases}
			\mcA\xi=\zeta f^m-\zeta \pt_ty^m-2w\nabla \zeta \cdot \nabla y^m-y^m\Div(w\nabla \zeta), &\mbox{in }\Om-B(0,\f{1}{2}R_0),\\
			\xi=0, &\mbox{on }\pt (\Om-B(0,\f{1}{2}R_0)), 
		\end{cases}
	\end{equation*}
	From \cite[Theorem 1 in Chapter 6.3.1 (p.~327) and Theorem 4 in Chapter 6.3.2 (p.~334)]{Evans}, we get
	\begin{equation*}
		\begin{split} 
		\|y^m\|_{H^2(\Om-B(0,R_0))}
		&=\|\xi\|_{H^2(\Om-B(0,R_0))}\\
		&\leq C\left(\|f^m\|_{L^2(\Om)}+\|\pt_ty^m\|_{L^2(\Om)}+\|y^m\|_{H_0^1(\Om;w)}\right), 
		\end{split} 
	\end{equation*}
		and then, combining \eqref{01.11.4} and \eqref{03.08.1}, we get
		\begin{equation}\label{03.08.2}
			\begin{split}
				\|y^m\|_{L^2(0,T; H^2(\Om-B(0,R_0)))}\leq C\left(\|f\|_{L^2(Q)}+\|y^0\|_{H_0^1(\Om;w)}\right), 
			\end{split}
		\end{equation}
			where the positive constant $C$ depends only on $\al, N, R_0, T$, and $\Om-B(0,\f{1}{2}R_0)$.
		Combining \eqref{03.08.2}, \eqref{03.08.3}, and \eqref{01.11.13}, we obtain \eqref{01.11.8}.  
\end{proof}

\begin{proposition}\label{03.08.P1}
	Let $y^0\in L^2(\Om)$ and $f\in L^2(Q)$. Then 
	\begin{equation*}
		y\in L^2(0,T; H_0^1(\Om;w))\cap H^1(0,T; H^{-1}(\Om;w)) 
	\end{equation*}
	is a weak solution of \eqref{12.14.1} with respect to $(y^0, f)$ if and only if 
	\begin{equation*}
		y\in L^2(0,T; H_0^1(\Om;w))
	\end{equation*}
	and for each $\psi\in C^\iy(\ol Q)$ with $\supp \psi(t) \s \Om$ for all $t\in [0,T]$ and $\psi(T)=0$ we get 
	\begin{equation}\label{03.08.4}
		-\iint_Q y\pt_t\psi\df x\df t+\iint_Q w\nabla y\cdot \nabla\psi\df x\df t=\iint_Q f\psi\df x\df t+\int_\Om y^0\psi(0)\df x. 
	\end{equation}
\end{proposition}

\begin{proof}
		We first prove the necessity.
	
		Let $\psi\in C^\iy(\ol Q)$ with $\supp \psi(t)\s \Om$ for all $t\in [0,T]$ and $\psi(T)=0$. Then, from Definition \ref{12.11.D1}, we get
	\begin{equation*}
		\lg \pt_ty, \psi(t)\rg_{H^{-1}(\Om;w),H_0^1(\Om;w)}+\int_\Om w\nabla y\cdot \nabla\psi(t)\df x=\int_\Om f\psi(t)\df x
	\end{equation*}
	for a.e.~$t\in [0,T]$. Integrating on $[0,T]$, we obtain 
	\begin{equation*}
		\int_0^T\lg \pt_ty,\psi\rg_{H^{-1}(\Om;w),H_0^1(\Om;w)}\df t+\iint_Q w\nabla y\cdot \nabla\psi\df x\df t=\iint_Q f\psi\df x\df t. 
	\end{equation*}
		Since $y$ is a weak solution of \eqref{12.14.1} with respect to $(y^0,f)$, we have $y\in C([0,T]; L^2(\Om))$. Integrating by parts, we obtain
	\begin{equation*}
		-\iint_Q y\pt_t\psi\df x\df t+\iint_Q w\nabla y\cdot \nabla\psi\df x\df t=\iint_Q f\psi\df x\df t+\int_\Om y(0)\psi(0)\df x. 
	\end{equation*}
		Together with Definition \ref{12.11.D1} (ii), this gives \eqref{03.08.4}.
	
		We next prove the sufficiency.
	
	From 
	\begin{equation*}
		\begin{split} 
		\lg\mcA y, \psi\rg_{L^2(0,T; H^{-1}(\Om;w)),L^2(0,T; H_0^1(\Om;w))}
		&=\int_0^T\lg \mcA y, \psi\rg_{H^{-1}(\Om;w),H_0^1(\Om;w)}\df t\\
		&=\iint_Q w\nabla y\cdot \nabla\psi\df x\df t
		\end{split} 
	\end{equation*}
		for all $\psi\in L^2(0,T; H_0^1(\Om;w))$, we have
	\begin{equation*}
		\|\mcA y\|_{L^2(0,T; H^{-1}(\Om;w))}\leq \|y\|_{L^2(0,T; H_0^1(\Om;w))}. 
	\end{equation*}
			Therefore,
	\begin{equation*}
		\mcA: L^2(0,T; H_0^1(\Om;w))\ra L^2(0,T; H^{-1}(\Om;w))
	\end{equation*}
			is a bounded linear operator. Since $f\in L^2(Q)\s L^2(0,T; H^{-1}(\Om;w))$, \eqref{03.08.4} gives
	\begin{equation*}
		-\iint_Q y\pt_t\psi\df x\df t+\iint_Q w\nabla y\cdot \nabla\psi\df x\df t=\iint_Q f\psi\df x\df t 
	\end{equation*}
		for all $\psi\in C_0^\iy(Q)$. This shows that
	\begin{equation*}
		\begin{split} 
		&\int_0^T \lg\pt_ty, \psi\rg_{H^{-1}(\Om;w),H_0^1(\Om;w)}\df t\\ 
		&=\int_0^T\lg f, \psi\rg_{H^{-1}(\Om;w),H_0^1(\Om;w)}\df t-\int_0^T\lg \mcA y, \psi\rg_{H^{-1}(\Om;w),H_0^1(\Om;w)}\df t,
		\end{split} 
	\end{equation*}
			and therefore $\pt_ty\in L^2(0,T; H^{-1}(\Om;w))$.
	
		We now prove that $y$ is a weak solution of \eqref{12.14.1} with respect to $(y^0,f)$.
	
		Let $t\in (0,T)$. Set $\de_0=\f{1}{4}\min\{t, T-t\}$. For each $\de\in (0,\de_0)$, choose $\zeta\in C_0^\iy(\R)$, $0\leq \zeta \leq 1$, such that $\zeta=1$ on $(t-\de,t+\de)$ and $\zeta=0$ on $(-\iy, t-\de-\e)\cup (t+\de+\e,+\iy)$ for each $\e\in (0,\de)$. Fix $v\in C_0^\iy(\Om)$. Then $\zeta v\in C_0^\iy(Q)$, and
	\begin{equation*}
		\begin{split}
			-\iint_Q yv\pt_t\zeta \df x\df t+\iint_Q \zeta  w\nabla y \cdot \nabla v\df x\df t=\iint_Q \zeta fv\df x\df t,
		\end{split}
	\end{equation*}
		that is,
	\begin{equation*}
		\int_0^T \zeta(t)\lg \pt_ty, v\rg_{H^{-1}(\Om;w),H_0^1(\Om;w)}\df t+\iint_Q \zeta w\nabla y \cdot \nabla v\df x\df t=\iint_Q \zeta fv\df x\df t. 
	\end{equation*}
	Letting $\e\ra 0^+$, we get
	\begin{equation*}
		\int_{t-\de}^{t+\de}\lg \pt_ty, v\rg_{H^{-1}(\Om;w),H_0^1(\Om;w)}\df t+\int_{t-\de}^{t+\de}\int_\Om w\nabla y\cdot \nabla v\df x\df t=\int_{t-\de}^{t+\de}\int_\Om fv\df x\df t. 
	\end{equation*}
		Letting $\de\ra 0^+$ and applying the Lebesgue Differentiation Theorem, we obtain
	\begin{equation*}
		\lg \pt_ty, v\rg_{H^{-1}(\Om;w),H_0^1(\Om;w)}+(y, v)_{H_0^1(\Om;w)}=(f,v)_{L^2(\Om)}. 
	\end{equation*}
			Since $C_0^\iy(\Om)$ is dense in $H_0^1(\Om;w)$, this proves Definition \ref{12.11.D1} (i).
	
	Let $\xi\in C^\iy(\R)$ with $\xi(T)=0$, and $v\in C_0^\iy(\Om)$. From \eqref{03.08.4}, we get
	\begin{equation*}
		-\iint_Q y v\pt_t\xi\df x\df t+\iint_Q \xi  w\nabla y\cdot \nabla v\df x\df t=\iint_Q \xi fv\df x\df t+\xi(0)\int_\Om y^0v\df x. 
	\end{equation*}
		Since $y\in L^2(0,T; H_0^1(\Om;w))$, write $y=\sum_{n=1}^\iy y_n(t)\Phi_n(x)$. Letting $v\ra \Phi_n$ for each $n\in\N$, we get
	\begin{equation*}
		-\int_0^T y_n(t) \pt_t\xi\df t+\la_n\int_0^T \xi y_n(t)\df t=\int_0^T f_n\xi\df t+\xi(0)y_n^0, 
	\end{equation*}
		where $y_n^0=(y^0,\Phi_n)_{L^2(\Om)}$ and $f_n=(f,\Phi_n)_{L^2(\Om)}$ for all $n\in\N$. This implies that $y_n\ (n\in\N)$ solves
	\begin{equation*}
		\begin{cases}
			\f{\df}{\df t}y_n+\la_ny_n=f_n,&\mbox{for } t\in [0,T],\\
			y_n(0)=y_n^0, 
		\end{cases}
	\end{equation*}
		From \eqref{01.11.1}, \eqref{01.11.2}, and the proof of Theorem \ref{01.11.T1}, we get $y^m=\sum_{n=1}^m y_n\Phi_n\ra z$ weakly in $L^2(0,T; H_0^1(\Om;w))$ as $m\ra\iy$; see \eqref{01.11.6}. Here $z$ is the weak solution of \eqref{12.14.1} with respect to $(y^0,f)$. On the other hand, $y^m\ra \sum_{n=1}^\iy y_n\Phi_n=y$ strongly in $L^2(0,T; H_0^1(\Om;w))$. Hence $y=z$. Moreover, by Definition \ref{12.11.D1} (ii), we have $y(0)=z(0)=y^0$. This completes the proof of the proposition.
\end{proof}

\begin{corollary}\label{01.12.C1}
		Let $y^0\in H_0^1(\Om;w)$ satisfy $\mcA y^0\in H_0^1(\Om;w)$. Let $y$ be the weak solution of \eqref{12.14.1} with respect to $(y^0,0)$. Then
	\begin{equation*}
		y\in L^2(0, T; D(\mcA^2)), \mbox{ with } \pt_t y\in L^2(0, T; D(\mcA)), 
	\end{equation*}
		and
	\begin{equation*}
		\begin{split}
			\esssup_{t\in [0,T]}\|\mcA y(t)\|_{H_0^1(\Om;w)}+\|\pt_ty\|_{L^2(0,T; H^2(\Om-B(0,R_0)))}+\|\pt_ty\|_{L^2(0, T; D(\mcA))}\leq C\|\mcA y^0\|_{H_0^1(\Om;w)},  
		\end{split}
	\end{equation*}
	where the constant $C>0$ depends only on $\al, T, N$ and $M$. 
\end{corollary}

\begin{proof}
		Recall that $y\in D(\mcA^2)$ if and only if $y\in D(\mcA)$ and $\mcA y\in D(\mcA)$.
		From Lemma \ref{06.28.L1}, we have $\mcA y^0=\sum_{n=1}^\iy (y^0, \Phi_n)_{L^2(\Om)}\la_n\Phi_n$. Note that
	\begin{equation*}
		\wh y=\sum_{n=1}^\iy (y^0,\Phi_n)_{L^2(\Om)}\la_ne^{-\la_nt}=\pt_t y, 
	\end{equation*}
		is the weak solution of \eqref{12.14.1} with respect to $(\mcA y^0,0)$. Therefore, \eqref{01.11.8} and the identity $\mcA \pt_ty=\pt_t\mcA y$ yield the corollary.
\end{proof}

\section{Shape Design}\label{S3}

In this section, we construct a sequence of uniformly parabolic approximate equations by means of the shape design method. The corresponding solutions approximate the weak solution of \eqref{12.14.1} in an appropriate sense. This approximation will be used in Section \ref{S4}.

Let $\de\in (0,\de_0)$, where $\de_0=\f{1}{16}\min\{1,\min\{R_0, M\}\}$. Let $\Om_\de\in C^3$ be as in Assumption \ref{Assumption (H)}. Then $\Om_\de-B(0,\f{1}{2}R_0)=\Om-B(0,\f{1}{2}R_0)$ for all $\de\in (0,\de_0)$.
Denote $Q_\de=\Om_\de\ts (0,T)$, $w_\de=w|_{\Om_\de}$, and
\begin{equation*}
	\mcA_\de y=-\Div(w_\de\nabla y).
\end{equation*}

We consider the equation
	\begin{equation}\label{01.11.14}
		\begin{cases}
			\pt_t y_\de+\mcA_\de y_\de=f_\de, &\mbox{in }Q_\de,\\
			y_\de=0, &\mbox{on }\pt Q_\de, \\
			y_\de(0)=y_\de^0, &\mbox{in }\Om_\de, 
		\end{cases}
	\end{equation}
where $y_\de^0\in L^2(\Om_\de)$ and $f_\de\in L^2(Q_\de)$. Equation \eqref{01.11.14} is uniformly parabolic.

As in Section \ref{S2}, we introduce the corresponding solution spaces and collect the needed results.

We define
\begin{equation*}
	H^{1}(\Omega_\de; w_\de) = \left\{u \in L^2(\Omega_\de) \colon \int_{\Om_\de}w_\de\nabla u\cdot \nabla u\df x<+\iy\right\}.
\end{equation*}
The inner product and norm on $H^1(\Omega_\de; w_\de)$ are defined by
\begin{equation*}
	(u, v)_{H^1(\Omega_\de; w_\de)} =\int_{\Om_\de} uv\df x+\int_{\Om_\de} (\nabla u\cdot \nabla v)w_\de\df x,\quad \|u\|_{H^1(\Om_\de;w_\de)}=(u,u)_{H^1(\Om_\de;w_\de)}^\f{1}{2}. 
\end{equation*} 
We also define
\begin{equation*}
	H_{0}^1(\Omega_\de; w_\de) = \text{the closure of } \mathcal{D}(\Omega_\de) \text{ in } H^1(\Omega_\de; w_\de)
\end{equation*} 
We denote by $H^{-1}(\Omega_\de; w_\de)$ the dual space of $H_{0}^1(\Omega_\de; w_\de)$ with pivot space $L^2(\Om_\de)$, which is a subspace of $\mathcal{D}'(\Omega_\de)$.  

Finally, we define
\begin{equation*}
	H^2(\Om_\de;w_\de)=\left\{u\in H^1(\Om_\de;w_\de)\colon \int_{\Om_\de} (\mcA_\de u)^2\df x<+\iy\right\}.
\end{equation*}
Its inner product and norm are defined by
\begin{equation*}
	(u,v)_{H^2(\Om_\de;w_\de)}=(u,v)_{H^1(\Om_\de;w_\de)}+\int_{\Om_\de} (\mcA_\de u)(\mcA_\de v)\df x, \quad \|u\|_{H^2(\Om_\de;w_\de)}=(u,u)_{H^2(\Om_\de;w_\de)}^\f{1}{2}. 
\end{equation*}

It is also well known that
\begin{equation*}
	H_0^1(\Om_\de;w_\de),\quad H^1(\Om_\de;w_\de), \mbox{ and } H^2(\Om_\de;w_\de)
\end{equation*}
are Hilbert spaces (see \cite{GC,Heinonen}), since $H_0^1(\Om_\de;w_\de)=H_0^1(\Om_\de)$, $H^1(\Om_\de;w_\de)=H^1(\Om_\de)$, and $H^2(\Om_\de;w_\de)=H^2(\Om_\de)$ by standard elliptic regularity. In general, however, $H^2(\Omega; w) \neq H^2(\Omega)$.

We denote 
\begin{equation*}
	D(\mcA_\de)=H^2(\Om_\de;w_\de)\cap H_0^1(\Om_\de;w_\de). 
\end{equation*}

\begin{lemma}\label{01.11.L1}
	Let $N \geq 2$ and $\alpha \in (0, 1)$. Then, for all $u \in H_0^1(\Omega_\de; w_\de)$, the following inequality holds:
	\begin{equation*}
		\int_{\Om_\de}|x|^{\al-2}u^2\df x\leq \f{4}{(N-2+\al)^2}\int_{\Om_\de} w_\de\nabla u\cdot \nabla u\df x.
	\end{equation*} 
\end{lemma}

\begin{proof}
	The proof follows exactly the same argument as in Lemma \ref{08.15.L1}.
\end{proof}

\begin{remark}\label{01.11.R1}
	From Lemma \ref{01.11.L1}, we have
	\begin{equation}\label{01.11.12}
		\int_{\Om_\de} u^2\df x \leq \f{4M^{2-\al}}{(N-2+\al)^2}\int_{\Om_\de} w_\de\nabla u\cdot \nabla u\df x,
	\end{equation}
	for all $u\in H_0^1(\Om_\de;w_\de)$. This is a Poincar\'{e} inequality.
	In particular, the norm
	\begin{equation}\label{01.11.15}
		\|u\|_{H_0^1(\Omega_\de; w_\de)} = \left(\int_{\Omega_\de} (\nabla u \cdot \nabla u) w_\de \, \mathrm{d}x\right)^{\frac{1}{2}}
	\end{equation}
	is an equivalent norm in $H_0^1(\Omega_\de; w_\de)$. Hereafter, we use \eqref{01.11.15} to define the norm of $H_0^1(\Omega_\de; w_\de)$.
\end{remark}

\begin{lemma}\label{01.11.L2}
	The embedding $H_0^1(\Omega_\de; w_\de) \hookrightarrow L^2(\Omega_\de)$ is compact.
\end{lemma}

\begin{proof}
	Since $H_0^1(\Om_\de;w_\de)=H_0^1(\Om_\de)$, the conclusion follows. 
\end{proof}

From Remark \ref{01.11.R1} and Lemma \ref{01.11.L1}, we deduce that the degenerate partial differential operator $\mcA_\de$ has discrete point spectrum
\begin{equation}\label{03.09.1}
	0<\la_1^\de< \la_2^\de\leq \la_3^\de\leq  \cdots \ra +\iy, 
\end{equation}
that is, $\la_n^\de\ (n\in\N)$ satisfies
\begin{equation}\label{03.09.2}
	\begin{cases}
		\mcA_\de  \Phi_n^\de=\la_n ^\de\Phi_n^\de, & \mbox{in }\Om_\de, \\
		\Phi_n^\de=0, &\mbox{in }\pt\Om_\de. 
	\end{cases}
\end{equation}
Moreover, from \eqref{01.11.12},  we get
\begin{equation}\label{03.09.3}
	\la_1^\de=\inf_{0\neq u\in H_0^1(\Om_\de;w_\de)}\f{\int_{\Om_\de} w_\de\nabla u\cdot \nabla u\df x}{\int_{\Om_\de} u^2\df x}\geq \f{(N-2+\al)^2}{4M^{2-\al}}.
\end{equation}

\begin{notation} \label{03.09.N1}
	We denote by $\Phi_n^\de(x)$ the $n$th eigenfunction of $\mcA_\de$ corresponding to the eigenvalue $\la_n^\de\ (n\in\N)$. Then $\{\Phi_n^\de\}_{n\in\N}$ is an orthonormal basis of $L^2(\Om_\de)$ and an orthogonal subset of $H_0^1(\Om_\de;w_\de)$. See \cite[Theorem 7 (p.~728) in Appendix D]{Evans} or the proof of Lemma \ref{03.09.L1}.
\end{notation}

\begin{lemma}\label{03.09.L1}
	Let $u=\sum_{i=1}^\iy u_i \Phi_i^\de\in H_0^1(\Om_\de;w_\de)$ with $u_i=(u,\Phi_i^\de)_{L^2(\Om_\de)}$ for all $i\in\N$. We have $\nabla u=\sum_{i=1}^\iy u_i\nabla \Phi_i^\de$ and $\|u\|_{H_0^1(\Om_\de;w_\de)}=(\sum_{i=1}^\iy u_i^2\la_i^\de)^\f{1}{2}$, and
	\begin{equation*}
		u\in H^2(\Om_\de;w_\de)\Lra \sum_{i=1}^\iy u_i^2(\la_i^\de)^2<\iy,
	\end{equation*}
	and
	\begin{equation*}
		\mcA_\de  u=\sum_{i=1}^\iy u_i\la_i^\de\Phi_i ^\de\mbox{ and } \|\mcA_\de u\|_{L^2(\Om_\de)}=\left(\sum_{i=1}^\iy u_i^2(\la_i^\de)^2\right)^\f{1}{2}.
	\end{equation*}
\end{lemma}

\begin{proof}
	The proof is the same as that of Lemma \ref{06.28.L1}.
\end{proof}

\begin{definition}\label{01.11.D1}
		We call a function
	\begin{equation*}
		y_\de\in L^2(0,T; H_0^1(\Om_\de;w_\de)), \mbox{ with } \pt_ty_\de\in L^2(0,T; H^{-1}(\Om_\de;w_\de))
	\end{equation*}
a weak solution of \eqref{01.11.14} with respect to $(y_\de^0,f_\de)$ if
	
		(i) for every $v\in H_0^1(\Om_\de;w_\de)$ and for a.e.~$t\in (0,T)$ we have
	\begin{equation*}
		\lg\pt_{t}y_\de, v\rg_{H^{-1}(\Om_\de;w_\de),H_0^1(\Om_\de;w_\de)}+\int_{\Om_\de} w_\de\nabla y_\de\cdot \nabla v\df x=\int_{\Om_\de} f_\de v\df x, 
	\end{equation*}
	
	(ii) $y_\de(0)=y_\de^0$. 
\end{definition}

\begin{theorem}\label{01.11.T2} 
		Let $\al\in (0,1)$, $y_\de^0\in L^2(\Om_\de)$, and $f_\de\in L^2(Q_\de)$. Then equation \eqref{01.11.14} with respect to $(y_\de^0,f_\de)$ admits a unique weak solution
	\begin{equation*}
		y_\de\in L^2(0,T; H_0^1(\Om_\de;w_\de)), \mbox{ with } \pt_ty_\de\in L^2(0,T; H^{-1}(\Om_\de;w_\de))
	\end{equation*} 
		and this solution satisfies the estimate
	\begin{equation}\label{01.11.16}
		\begin{split} 
			&\esssup_{t\in [0,T]}\|y_\de(t)\|_{L^2(\Om_\de)}+\|y_\de\|_{L^2(0,T;H_0^1(\Om_\de;w_\de))}+\|\pt_ty_\de\|_{L^2(0,T; H^{-1}(\Om_\de;w_\de))}\\
			&\leq C\left(\|f_\de\|_{L^2(Q_\de)}+\|y_\de^0\|_{L^2(\Om_\de)}\right), 
		\end{split} 
	\end{equation}
		where the constant $C>0$ depends only on $\al, N, T$, and $M$.
	
		Moreover, if we further assume that $y_\de^0\in H_0^1(\Om_\de;w_\de)$, then the weak solution belongs to
	\begin{equation*}
		y_\de\in L^2(0,T; D(\mcA_\de))\cap C([0,T]; H_0^1(\Om_\de;w_\de)), \mbox{ with } \pt_ty_\de\in L^2(Q_\de), 
	\end{equation*}
		and satisfies the estimate
	\begin{equation}\label{01.11.17}
		\begin{split} 
			&\esssup_{t\in [0,T]}\|y_\de(t)\|_{H_0^1(\Om_\de;w_\de)}+\|y_\de\|_{L^2(0,T; H^2(\Om-B(0,R_0)))}+\|y_\de\|_{L^2(0,T; D(\mcA_\de))}+\|\pt_ty_\de\|_{L^2(Q_\de)}\\
			&\leq C\left(\|f_\de\|_{L^2(Q_\de)}+\|y_\de^0\|_{H_0^1(\Om_\de;w_\de)}\right), 
		\end{split} 
	\end{equation}
		where the constant $C>0$ depends only on $\al, N, R_0, T, M$, and $\Om-B(0,\f{1}{2}R_0)$.
\end{theorem}

\begin{proof}
	The argument follows the proof of Theorem \ref{01.11.T1}.
	
	The estimates \eqref{01.11.16} and \eqref{01.11.17} are analogous to \eqref{01.11.7} and \eqref{01.11.8}, respectively. Note that the derivation of \eqref{01.11.7} and \eqref{01.11.8} relies only on \eqref{01.11.4}, \eqref{01.11.5}, \eqref{03.08.1}, \eqref{01.11.13}, \eqref{03.08.3}, and \eqref{03.08.2}. These estimates depend only on Lemma \ref{08.15.L1}, Remark \ref{08.16.R1}, Lemma \ref{06.28.L1}, and \cite[Theorem 1 in Chapter 6.3.1 (p.~327) and Theorem 4 in Chapter 6.3.2 (p.~334)]{Evans}.
	
	By replacing Lemma \ref{08.15.L1} with Lemma \ref{01.11.L1}, Remark \ref{08.16.R1} with Remark \ref{01.11.R1}, and Lemma \ref{06.28.L1} with Lemma \ref{01.11.L2}, and $\Om_\de-B(0,\f{1}{2}R_0)=\Om-B(0,\f{1}{2}R_0)$ for all $\de\in (0,\de_0)$, we obtain the estimates \eqref{01.11.16} and \eqref{01.11.17}.
\end{proof}

\begin{proposition}\label{03.08.P2}
		Let $y_\de^0\in L^2(\Om_\de)$ and $f_\de\in L^2(Q_\de)$. Then a function
	\begin{equation*}
		y_\de\in L^2(0,T; H_0^1(\Om_\de;w_\de))\cap H^1(0,T; H^{-1}(\Om_\de;w_\de)) 
	\end{equation*}
	is a weak solution of \eqref{01.11.14} with respect to $(y_\de^0, f_\de)$ if and only if
	\begin{equation*}
		y_\de\in L^2(0,T; H_0^1(\Om_\de;w_\de))
	\end{equation*}
	and for each $\psi\in C^\iy(\ol Q_\de)$ with $\supp \psi(t) \s \Om_\de$ for all $t\in [0,T]$ and $\psi(T)=0$ we get 
	\begin{equation}\label{03.09.4}
		-\iint_{Q_\de} y_\de\pt_t\psi\df x\df t+\iint_{Q_\de} w_\de\nabla y_\de\cdot \nabla\psi\df x\df t=\iint_{Q_\de} f_\de\psi\df x\df t+\int_{\Om_\de} y_\de^0\psi(0)\df x. 
	\end{equation}
\end{proposition}

\begin{proof}
	The proof is similar to that of Proposition \ref{03.08.P1}.
\end{proof}

\begin{corollary}\label{01.12.C2}
	Let $y_\de^0\in H_0^1(\Om_\de;w_\de)$ satisfy $\mcA_\de y_\de^0\in H_0^1(\Om_\de;w_\de)$. Let $y_\de$ be the weak solution of \eqref{01.11.14} with respect to $(y_\de^0,0)$. Then
	\begin{equation*}
		y_\de\in L^2(0, T; D(\mcA_\de^2)), \mbox{ with } \pt_t y_\de\in L^2(0, T; D(\mcA_\de)), 
	\end{equation*}
	and 
	\begin{equation*}
		\begin{split}
			&\esssup_{t\in [0,T]}\|\mcA_\de y_\de(t)\|_{H_0^1(\Om_\de;w_\de)}+\|\pt_ty_\de\|_{L^2(0,T; H^2(\Om-B(0,R_0)))}+\|\pt_ty_\de\|_{L^2(0, T; D(\mcA_\de))}\\
			&\leq C\|\mcA y_\de^0\|_{H_0^1(\Om_\de;w_\de)},  
		\end{split}
	\end{equation*}
	where the constant $C>0$ depends only on $\al, T, N$ and $M$. 
\end{corollary}

\begin{proof}
	The proof is the same as that of Corollary \ref{01.12.C1}.
\end{proof}

We now introduce the shape design framework for equation \eqref{01.11.14}.

\begin{definition}\label{01.11.D2}
For each $y_\de\in L^2(Q_\de)$, we denote
\begin{equation*}
	Ey_\de=
	\begin{cases}
		y_\de, &\mbox{on } Q_\de, \\
		0, &\mbox{on } Q-Q_\de.
	\end{cases}
\end{equation*}
Then $Ey_\de\in L^2(Q)$ and $\|Ey_\de\|_{L^2(Q)}=\|y_\de\|_{L^2(Q_\de)}$. If $y_\de\in L^2(0,T; H_0^1(\Om_\de;w_\de))$, then 
\begin{equation*}
	Ey_\de\in L^2(0,T; H_0^1(\Om;w)), \mbox{ and } \|Ey_\de\|_{L^2(0,T; H_0^1(\Om;w))}=\|y_\de\|_{L^2(0,T; H_0^1(\Om_\de;w_\de))}. 
\end{equation*}
If $y_\de\in H^1(0,T; L^2(\Om_\de))$, then
\begin{equation*}
	Ey_\de\in H^1(0,T; L^2(\Om)), \pt_tEy_\de=E(\pt_ty_\de),\mbox{ and } \|Ey_\de\|_{H^1(0,T; L^2(\Om))}=\|y_\de\|_{H^1(0,T; L^2(\Om_\de))}. 
\end{equation*} 
\end{definition}

The following theorem is one of the main results of the paper.

\begin{theorem}\label{01.11.T3}
	Assume that Assumption \ref{Assumption (H)} holds. Let $\al\in (0,1)$, $y^0\in C_0^\iy(\Om)$, and $f\in L^2(Q)$. Let $\de_0=\f{1}{16}\min\left\{1,R_0,M\right\}$.
	Let $y$ be the weak solution of \eqref{12.14.1} with respect to $(y^0,f)$, and let $y_\de\ (0<\de<\de_0)$ be the weak solution of \eqref{01.11.14} with respect to $(y^0|_{\Om_\de}, f|_{Q_\de})$. Then, up to a subsequence, we have
	\begin{equation}\label{01.11.18}
		\begin{split}
			Ey_\de
			&\ra y \hspace{3.5mm} \mbox{ weakly in } L^2(0,T; H_0^1(\Om;w)), \\
			\pt_t Ey_\de
			&\ra \pt_ty \mbox{ weakly in } L^2(Q),\\
			Ey_\de
			&\ra y\hspace{3.5mm} \mbox{ strongly in } L^2(Q). 
		\end{split}
	\end{equation}
	Moreover, if we additionally assume that $f=0$, then, up to a subsequence, we get
	\begin{equation}\label{01.12.5}
		\f{\pt Ey_\de}{\pt \nu}=\f{\pt y_\de}{\pt \nu}\ra \f{\pt y}{\pt \nu} \mbox{ strongly in } L^2(0,T; L^2(\Ga^+)). 
	\end{equation}
\end{theorem}

\begin{proof}
		We prove this theorem in the following steps.
	
	{\it Step 1}. We prove \eqref{01.11.18}. 
	
		By a standard subsequence argument, it suffices to prove the case
		\begin{equation*}
		\de_0=\min\left\{\f{1}{16}, \f{1}{16}R_0, \f{1}{16}M, \dist(\supp y^0, \pt\Om)\right\}. 
	\end{equation*}
		For all $0<\de<\de_0$, from $\Om_\de \s \Om$, $y^0|_{\Om_\de}\in C_0^\iy(\Om_\de;w_\de)$, \eqref{01.11.17}, and Definition \ref{01.11.D2}, we obtain
	\begin{equation*}
		\begin{split} 
			&\esssup_{t\in [0,T]}\|Ey_\de(t)\|_{H_0^1(\Om;w)}+\|\pt_tEy_\de\|_{L^2(Q)}\leq C\left(\|f\|_{L^2(Q)}+\|y^0\|_{H_0^1(\Om;w)}\right), 
		\end{split} 
	\end{equation*}
		where the constant $C>0$ depends only on $\al, N, T$, and $M$, and is independent of $\de\in (0,\de_0)$. Therefore, there exists
	\begin{equation}\label{03.09.5}
		z\in L^2(0,T; H_0^1(\Om;w))\cap H^1(0,T; L^2(\Om))
	\end{equation}
	such that 
	\begin{equation}\label{01.11.19}
		\begin{split} 
		Ey_\de
		&\ra z \mbox{ weakly in } L^2(0,T;H_0^1(\Om;w)),\\
		\pt_t Ey_\de 
		&\ra \pt_t z \mbox{ weakly in } L^2(Q). 
		\end{split} 
	\end{equation} 
			Note that $y_\de$ is the weak solution of \eqref{01.11.14} with respect to $(y^0|_{\Om_\de}, f|_{Q_\de})$. Let $\psi\in C^\iy(\ol Q)$ satisfy $\supp\psi(t)\s \Om$ for all $t\in [0,T]$ and $\psi(T)=0$. Since $\supp\psi$ is compactly contained in $\Om\ts [0,T]$, there exists $\de_\psi\in (0,\de_0)$ such that $\supp\psi(t)\s \Om_\de$ for all $t\in [0,T]$ whenever $0<\de<\de_\psi$. Hence, by Proposition \ref{03.08.P2}, we get
	\begin{equation*}
		-\iint_{Q_\de} y_\de \pt_t\psi\df x\df t+\iint_{Q_\de} w_\de \nabla y_\de \cdot \nabla\psi\df x\df t=\iint_{Q_\de} f_\de \psi\df x\df t+\int_{\Om_\de} y^0 \psi(0)\df x.
	\end{equation*} 
		This shows that
	\begin{equation*}
		-\iint_{Q} (Ey_\de) \pt_t\psi\df x\df t+\iint_{Q} w \nabla Ey_\de \cdot \nabla\psi\df x\df t=\iint_{Q} f \psi\df x\df t+\int_\Om y^0\psi(0)\df x.
	\end{equation*}
		Combined with \eqref{01.11.19}, this yields
	\begin{equation*}
		-\iint_Q z\pt_t\psi\df x\df t+\iint_Q w\nabla z\cdot \nabla\psi\df x\df t=\iint_Q f\psi\df x\df t+\int_\Om y^0\psi(0)\df x. 
	\end{equation*}
			Hence, by Proposition \ref{03.08.P1} and \eqref{03.09.5}, $z$ is a weak solution of \eqref{12.14.1} with respect to $(y^0,f)$. By uniqueness, we must have $z=y$.
	
			Moreover, since the embeddings $H_0^1(\Om;w)\hra L^2(\Om)$ and
	\begin{equation*}
		\left\{z\in L^2(0,T; H_0^1(\Om;w))\colon \pt_t z\in L^2(Q)\right\}\hra L^2(Q)
	\end{equation*}
		are compact, we obtain the third convergence in \eqref{01.11.18}.
	
		{\it Step 2}. We prove \eqref{01.12.5}. 
		
			Since $y^0\in C_0^\iy(\Om)$, we have $y^0\in D(\mcA^2)$, that is, $y^0\in D(\mcA)$ and $\mcA y^0\in D(\mcA)$.
		
		From \eqref{01.11.17} we get
		\begin{equation*}
			\begin{split}
				\|Ey_\de\|_{L^2(0,T; H^2(\Om-B(0,R_0)))}=\|y_\de\|_{L^2(0,T; H^2(\Om-B(0,R_0)))}\leq C\|y^0\|_{H_0^1(\Om_\de;w_\de)}, 
			\end{split}
		\end{equation*}
			Then
		\begin{equation}\label{03.09.6}
			\left\|\f{\pt y_\de}{\pt\nu}\right\|_{L^2(0,T; H^\f{1}{2}(\pt\Om-B(0,R_0)))}\leq C\|y^0\|_{H_0^1(\Om;w)}
		\end{equation}
		by the classical Sobolev trace theorem, 
			where the positive constant $C$ depends only on $\al, N, R_0, T, M$, and $\Om-B(0,\f{1}{2}R_0)$.
			Therefore, there exists a function $z\in L^2(0,T; H^\f{1}{2}(\pt\Om-B(0,R_0)))$ such that, up to a subsequence,
		\begin{equation}\label{03.09.8}
			\f{\pt y_\de}{\pt\nu}\ra z \mbox{ weakly in }L^2(0,T; H^\f{1}{2}(\pt\Om-B(0,R_0))).
		\end{equation}
			Let $\xi\in C_0^\iy((\pt\Om-B(0,R_0))\ts (0,T))$, and extend it to a function in $C_0^\iy((\ol\Om-B(0,\f{1}{2}R_0))\ts (0,T))$, still denoted by $\xi$. By \eqref{01.11.18}, we know that $Ey_\de\ra y$ weakly in $L^2(0,T;H_0^1(\Om;w))$ and $\pt_tEy_\de\ra \pt_ty$ weakly in $L^2(Q)$. Since $f=0$, we have $\mcA_\de y_\de=-\pt_ty_\de$. Consequently,
		\begin{equation*}
			\begin{split}
				\iint_{(\pt\Om-B(0,R_0))\ts (0,T)} \f{\pt y_\de}{\pt\nu}\xi\df S\df t
				&=\iint_{Q_\de} (\mcA_\de y_\de)\xi\df x\df t+\iint_{Q_\de} w_\de\nabla y_\de\cdot \nabla\xi\df x\df t\\
				&=-\iint_Q (\pt_tEy_\de)\xi\df x\df t+\iint_Q w\nabla(Ey_\de)\cdot \nabla\xi\df x\df t\\
				&\ra -\iint_Q (\pt_ty)\xi\df x\df t+\iint_Q w\nabla y\cdot \nabla\xi\df x\df t\\
					&=\iint_{(\pt\Om-B(0,R_0))\ts (0,T)} \f{\pt y}{\pt\nu}\xi\df S\df t.
				\end{split}
			\end{equation*}
				Hence the weak limit is exactly $\f{\pt y}{\pt\nu}$, that is,
		\begin{equation}\label{03.09.7}
			\f{\pt y_\de}{\pt\nu}\ra \f{\pt y}{\pt\nu} \mbox{ weakly in }L^2(0,T; H^\f{1}{2}(\pt\Om-B(0,R_0))).
		\end{equation}
		
		Now, from Corollary \ref{01.12.C2}, we get 
		\begin{equation*}
			\|\pt_ty_\de\|_{L^2(0,T; H^2(\Om-B(0,R_0)))}\leq C\|\mcA y^0\|_{H_0^1(\Om;w)}, 
	\end{equation*}
		Thus
		\begin{equation*}
			\left\|\pt_t\f{\pt y_\de}{\pt\nu}\right\|_{L^2(0,T; L^2(\pt\Om-B(0,R_0)))}=\left\|\f{\pt (\pt_ty_\de)}{\pt \nu}\right\|_{L^2(0,T; L^2(\pt\Om-B(0,R_0)))}\leq C\|\mcA y^0\|_{H_0^1(\Om;w)}, 
		\end{equation*}
			where the positive constant $C$ depends only on $\al, N, R_0, T, M$, and $\Om-B(0,\f{1}{2}R_0)$. Combining this with \eqref{03.09.6}, the compact embedding $H^\f{1}{2}(\pt\Om-B(0,R_0))\hra L^2(\pt\Om-B(0,R_0))$, and
	\begin{equation*}
		\begin{split} 
		&\left\{z\in L^2(0,T; H^\f{1}{2}(\pt\Om-B(0,R_0)))\colon \pt_tz\in L^2(0,T; L^2(\pt\Om-B(0,R_0)))\right\}\\
			&\hra L^2(0,T; L^2(\pt\Om-B(0,R_0)))
			\end{split} 
		\end{equation*}
			we conclude that, up to a subsequence,
		\begin{equation*}
			\f{\pt y_\de}{\pt\nu}\ra g \mbox{ strongly in }L^2(0,T;L^2(\pt\Om-B(0,R_0)))
		\end{equation*}
		for some $g\in L^2(0,T;L^2(\pt\Om-B(0,R_0)))$. In view of \eqref{03.09.7}, this strong limit must coincide with $\f{\pt y}{\pt\nu}$. Therefore,
		\begin{equation*}
			\f{\pt y_\de}{\pt\nu}\ra \f{\pt y}{\pt\nu} \mbox{ strongly in }L^2(0,T;L^2(\pt\Om-B(0,R_0))).
		\end{equation*}
		Since $\Ga^+\s \pt\Om-B(0,R_0)$ by Assumption \ref{Assumption (H)}, we obtain \eqref{01.12.5}.
	
		This completes the proof.
\end{proof}

\section{Observability on the Boundary}\label{S4}

In this section, we establish boundary observability for the backward equation associated with \eqref{12.14.1}. We first prove a boundary observability estimate for the approximate parabolic equations \eqref{01.06.19}, and then pass to the limit.

\subsection{Boundary Observability for Approximate Parabolic Equations}

Let $\de\in (0,\de_0)$.
We consider the backward degenerate parabolic equation
\begin{equation}\label{01.06.19}
	\begin{cases}
		\pt_t\vp_\de-\mcA_\de \vp_\de=f_\de, &\mbox{in }Q_\de, \\
		\vp_\de=0, &\mbox{on }\pt Q_\de,\\ 
		\vp_\de(T)=\vp_\de^T, &\mbox{in }\Om_\de, 
	\end{cases}
\end{equation}
where $\vp_\de^T\in L^2(\Om_\de)$ and $f_\de\in L^2(Q_\de)$. For each $\de\in (0,\de_0)$, equation \eqref{01.06.19} is uniformly backward parabolic.

Choose
\begin{equation}\label{01.11.22}
	\eta(x)=|x|^{2-\al}. 
\end{equation}
Let 
\begin{equation}\label{01.11.23}
	\Theta=\f{1}{[t(T-t)]^4}, \quad \xi=\Theta (\ga-\eta) 
\end{equation}
with $\ga=|\eta|_{L^\iy(\Om)}+1$,
and let
	\begin{equation}\label{01.06.18}
		\psi=e^{-s\xi}\vp_\de\ \Lra \ \vp_\de=e^{s\xi}\psi,
	\end{equation}
	where $\vp_\de$ is the weak solution of \eqref{01.06.19} with respect to $(\vp_\de^T,f_\de)$.
	
By \eqref{01.06.18}, we have
\begin{equation}\label{01.11.28}
	\psi(t)=\nabla\psi(t)=0 \mbox{ at } t=0,T, 
\end{equation}
and 
\begin{equation}\label{01.11.29}
	\psi=\pt_t\psi=0 \mbox{ on }\pt Q_\de.
\end{equation}
Moreover, from \eqref{01.11.23}, we have
\begin{equation}\label{01.11.24}
	\begin{split}
		|\Theta'|\leq C\Theta^\f{5}{4},\ |\Theta''|\leq C\Theta^\f{3}{2}, \ \max\{\Theta^2,\Theta^\f{9}{4}\}\leq C\Theta^3, \mbox{ and } |\xi_t|\leq C\ga \Theta^\f{5}{4}, \ |\xi_{tt}|\leq C\ga \Theta^\f{3}{2},
	\end{split}
\end{equation}
where the constant $C>0$ depends only on $T$. From \eqref{01.11.23}, we also obtain
\begin{equation}\label{01.11.27}
	\begin{split}
		\nabla\xi=-\Theta \nabla\eta,\quad \nabla\xi_t=\pt_t\nabla\xi=-\Theta'\nabla\eta, 
	\end{split}
\end{equation}
and
\begin{equation}\label{01.12.3}
	\begin{split} 
	 &\nabla\eta=(2-\al)|x|^{-\al}x, \quad w\nabla\eta=(2-\al)x,\quad  w\nabla\eta\cdot\nabla\eta =(2-\al)^2|x|^{2-\al}, \\
	 &\Div(w\nabla\eta)=(2-\al)N,\ \nabla [\Div(w\nabla\eta)]=0,\  w\nabla\eta\cdot \nabla (w\nabla\eta\cdot\nabla\eta)=(2-\al)^4|x|^{2-\al},\\
	 &\f{\pt^2\eta}{\pt x_i\pt x_j}=-(2-\al)\al |x|^{-\al-2}x_ix_j+(2-\al)|x|^{-\al}\de_{ij}, 
	 \end{split} 
\end{equation} 
where $\de_{ij}=1$ for $i=j$ and $\de_{ij}=0$ for $i\neq j$, 
and 
\begin{equation}\label{01.12.2}
	w\nabla \eta \cdot \nu=(2-\al)x\cdot\nu \mbox{ on }\pt\Om_\de, \mbox{ for all } \de\in (0, \de_0). 
\end{equation}

From \eqref{01.11.23}, we obtain
\begin{equation}\label{01.11.25}
	\begin{split}
		e^{-s\xi}f_\de=P_1\psi+P_2\psi, 
	\end{split}
\end{equation}
where 
\begin{equation}\label{01.11.26}
	\begin{split}
		P_1\psi
		&\equiv \sum_{i=1}^3P_{1i}\psi\equiv \psi_t+2sw\nabla\psi\cdot\nabla\xi+s \psi\Div(w\nabla\xi),\\
		P_2\psi
		&\equiv \sum_{i=1}^3P_{2i}\psi\equiv \Div(w\nabla\psi)+s\psi\xi_t+s^2\psi w\nabla\xi\cdot\nabla\xi. 
	\end{split}
\end{equation}
Here and below, for simplicity of notation, we often omit the subscript $\de$. For example, we write $w$ for $w_\de$, $\psi$ for $\psi_\de$, and similarly for other quantities, while the integration domains remain $\Om_\de$ and $Q_\de$.

\subsubsection{Computation}

Since \eqref{01.06.19} is uniformly parabolic for each $\de \in (0,\de_0)$, the following integrations by parts are justified. We emphasize, however, that the integration by parts involving the boundary term on $\partial Q_\de$ is not available in the limit case $\de=0$.

(1) We compute $(P_{11}\psi, P_2\psi)_{L^2(Q_\de)}$. 

Using \eqref{01.11.28}, \eqref{01.11.29}, and \eqref{01.11.27}, we obtain
\begin{equation*}
	\begin{split}
		(P_{11}\psi, P_2\psi)_{L^2(Q_\de)}
		&=-\f{s}{2}\iint_{Q_\de}\psi^2\xi_{tt}\df x\df t-s^2\iint_{Q_\de} \psi^2w\nabla\xi\cdot\nabla\xi_t\df x\df t\\
		&=-\f{s}{2}\iint_{Q_\de}\psi^2\xi_{tt}\df x\df t-s^2\iint_{Q_\de}\Theta \Theta'\psi^2w\nabla\eta\cdot\nabla\eta\df x\df t. 
	\end{split}
\end{equation*}

(2) We compute $(P_{12}\psi+P_{13}\psi, P_{21}\psi)_{L^2(Q_\de)}$. 

Using \eqref{01.11.29} and \eqref{01.11.27}, we obtain
\begin{equation*}
	\begin{split}
		&2s\iint_{Q_\de} \Theta w^2\nabla\psi\cdot [(D^2\psi)\nabla\eta]\df x=s\iint_{Q_\de} \Theta w^2\nabla |\nabla\psi|^2\cdot \nabla\eta\df x\df t\\
		&=s\iint_{\pt Q_\de} \Theta (w\nabla\psi\cdot\nabla\psi) (w\nabla\eta\cdot\nu)\df S\df t\\
		&\hspace{4.5mm}-s\iint_{Q_\de} \Theta (w\nabla\psi\cdot\nabla\psi)(\nabla w\cdot\nabla\eta)\df x\df t-s\iint_{Q_\de} \Theta (w\nabla\psi\cdot\nabla\psi)\Div(w\nabla\eta)\df x\df t, 
	\end{split}
\end{equation*}
and then 
\begin{equation*}
	\begin{split}
		&(P_{12}\psi, P_{21}\psi)_{L^2(Q_\de)}\\
		&=2s\iint_{\pt Q_\de}(w\nabla\psi\cdot\nu)(w\nabla\psi\cdot \nabla\xi)\df S\df t-2s\iint_{Q_\de}w\nabla\psi\cdot\nabla\left(w\nabla\psi\cdot\nabla\xi\right)\df x\df t\\
		&=-2s\iint_{\pt Q_\de}\Theta  (w\nabla\psi\cdot\nu)(w\nabla\psi\cdot\nabla\eta)\df S\df t+2s\iint_{Q_\de}\Theta w\nabla\psi\cdot \nabla (w\nabla\psi\cdot\nabla\eta)\df x\df t\\
		&=-2s\iint_{\pt Q_\de}\Theta  (w\nabla\psi\cdot\nu)(w\nabla\psi\cdot\nabla\eta)\df S\df t+2s\iint_{Q_\de}\Theta (w\nabla\psi\cdot \nabla w)(\nabla\psi\cdot\nabla\eta)\df x\df t\\
		&\hspace{4.5mm}+2s\iint_{Q_\de}\Theta w^2\nabla\psi\cdot[(D^2\eta)\nabla\psi]\df x\df t+2s\iint_{Q_\de} \Theta w^2\nabla\psi\cdot [(D^2\psi)\nabla\eta]\df x\df t\\
		&=-s\iint_{\pt Q_\de}\Theta  (w\nabla\psi\cdot\nabla\psi)(w\nabla\eta\cdot\nu)\df S\df t\\
		&\hspace{4.5mm}+2s\iint_{Q_\de}\Theta (w\nabla\psi\cdot \nabla w)(\nabla\psi\cdot\nabla\eta)\df x\df t+2s\iint_{Q_\de}\Theta w^2\nabla\psi\cdot[(D^2\eta)\nabla\psi]\df x\df t\\
		&\hspace{4.5mm}-s\iint_{Q_\de} \Theta (w\nabla\psi\cdot\nabla\psi)(\nabla w\cdot\nabla\eta)\df x\df t-s\iint_{Q_\de} \Theta (w\nabla\psi\cdot\nabla\psi)\Div(w\nabla\eta)\df x\df t
	\end{split}
\end{equation*}
by \eqref{01.11.29} and  \eqref{01.11.27}. 
Using again \eqref{01.11.29} and \eqref{01.11.27}, we obtain
\begin{equation*}
	\begin{split}
		&(P_{13}\psi, P_{21}\psi)_{L^2(Q_\de)}\\
		&=s\iint_{Q_\de}\Theta(w\nabla\psi\cdot\nabla\psi)\Div(w\nabla\eta)\df x\df t+s\iint_{Q_\de}\Theta \psi w\nabla\psi\cdot\nabla [\Div(w\nabla\eta)]\df x\df t. 
	\end{split}
\end{equation*}
Therefore,
\begin{equation*}
	\begin{split}
		&(P_{12}\psi+P_{13}\psi, P_{21}\psi)_{L^2(Q_\de)}\\
		&=-s\iint_{\pt Q_\de}\Theta  (w\nabla\psi\cdot\nabla\psi)(w\nabla\eta\cdot\nu)\df S\df t\\
		&\hspace{4.5mm}+2s\iint_{Q_\de}\Theta (w\nabla\psi\cdot \nabla w)(\nabla\psi\cdot\nabla\eta)\df x\df t+2s\iint_{Q_\de}\Theta w^2\nabla\psi\cdot[(D^2\eta)\nabla\psi]\df x\df t\\
		&\hspace{4.5mm}-s\iint_{Q_\de}\Theta (w\nabla\psi\cdot \nabla\psi)(\nabla w\cdot \nabla\eta)\df x\df t+s\iint_{Q_\de}\Theta \psi w\nabla\psi\cdot\nabla [\Div(w\nabla\eta)]\df x\df t. 
	\end{split}
\end{equation*}

\begin{remark}\label{01.08.R2}
		The term
	\begin{equation*} 
		-s\iint_{\pt Q_\de}\Theta  (w\nabla\psi\cdot\nabla\psi)(w\nabla\eta\cdot\nu)\df S\df t
	\end{equation*} 
		is well defined for each $\delta \in (0, \delta_0)$. However, when $\delta = 0$, it becomes difficult to obtain
	\begin{equation*}
		\f{\pt \psi}{\pt \nu}\in L^2([\pt\Om\cap B(0,\de_0)]\ts (0,T)),
	\end{equation*}
		since this would require improved boundary regularity of $\psi$ near $0\in\R^N$.
		At present, we do not know how to obtain such regularity.
\end{remark}

(3) We compute $(P_{12}\psi+P_{13}\psi,P_{22}\psi)_{L^2(Q_\de)}$. 

Using \eqref{01.11.29} and \eqref{01.11.27}, we obtain
\begin{equation*}
	\begin{split}
		&(P_{12}\psi+P_{13}\psi,P_{22}\psi)_{L^2(Q_\de)}=-s^2\iint_{Q_\de}\psi^2w\nabla\xi_t\cdot \nabla\xi\df x\df t=-s^2\iint_{Q_\de}\Theta\Theta'\psi^2w\nabla\eta\cdot\nabla\eta\df x\df t. 
	\end{split}
\end{equation*}

(4) We compute $(P_{12}\psi+P_{13}\psi, P_{23}\psi)_{L^2(Q_\de)}$. 

Using \eqref{01.11.29} and \eqref{01.11.27}, we obtain
\begin{equation*}
	\begin{split}
		&(P_{12}\psi+P_{13}\psi, P_{23}\psi)_{L^2(Q_\de)}\\
		&=-s^3\iint_{Q_\de} \psi^2w\nabla\xi\cdot \nabla(w\nabla\xi\cdot\nabla\xi)\df x\df t=s^3\iint_{Q_\de} \Theta^3\psi^2w\nabla\eta\cdot \nabla(w\nabla\eta\cdot\nabla\eta)\df x\df t.
	\end{split}
\end{equation*}

\subsubsection{Estimate}

From \eqref{01.11.24} and \eqref{01.12.3}, we obtain
\begin{equation*}
	\begin{split}
		\left|(P_{11}\psi, P_2\psi)_{L^2(Q_\de)}\right|\leq Cs\ga \iint_{Q_\de}\Theta^\f{3}{2}\psi^2\df x\df t+Cs^2\iint_{Q_\de}\Theta^\f{9}{4}\psi^2|x|^{2-\al}\df x\df t,
	\end{split}
\end{equation*}
where the constant $C>0$ depends only on $\al$ and $T$.
Moreover, \eqref{01.12.3} and \eqref{01.12.2} imply
\begin{equation*}
	\begin{split}
		&(P_{12}\psi+P_{13}\psi, P_{21}\psi)_{L^2(Q_\de)}\\
		&\geq -(2-\al)s\iint_{\Ga^+\ts (0,T)}\Theta (w\nabla\psi\cdot\nabla\psi)(x\cdot\nu)\df S\df t+(2-\al)^2 s\iint_{Q_\de}\Theta (w\nabla\psi\cdot\nabla\psi)\df x\df t;
	\end{split}
\end{equation*}
Furthermore, \eqref{01.11.24} yields
\begin{equation*}
	\left|(P_{12}\psi+P_{13}\psi, P_{22}\psi)_{L^2(Q_\de)}\right|\leq Cs^2\iint_{Q_\de}\Theta^\f{9}{4}\psi^2 |x|^{2-\al}\df x\df t,
\end{equation*}
where the constant $C>0$ depends only on $\al$ and $T$.
Finally,
\begin{equation*}
	\begin{split}
		(P_{12}\psi+P_{13}\psi, P_{23}\psi)_{L^2(Q_\de)}=(2-\al)^4s^3\iint_{Q_\de} \Theta^3\psi^2|x|^{2-\al}\df x\df t.
	\end{split}
\end{equation*}
Choosing $s_0\geq 1$ sufficiently large, depending only on $\al, N, \ga$, and $T$, we obtain for all $s\geq s_0$ that
\begin{equation*}
	\begin{split}
		&(P_1\psi,P_2\psi)_{L^2(Q_\de)}\\
		&\geq -(2-\al)s\iint_{\Ga^+\ts (0,T)}\Theta (w\nabla\psi\cdot\nabla\psi)(x\cdot\nu)\df S\df t\\
		&\hspace{4.5mm}+(2-\al)^2 s\iint_{Q_\de}\Theta (w\nabla\psi\cdot\nabla\psi)\df x\df t+(2-\al)^4s^3\iint_{Q_\de} \Theta^3\psi^2|x|^{2-\al}\df x\df t
	\end{split}
\end{equation*}
where we also used Lemma \ref{01.11.L1} and \eqref{01.11.29}, to infer that
\begin{equation*}
	\begin{split}
		s\ga\iint_{Q_\de} \Theta^\f{3}{2} \psi^2\df x\df t
		&\leq s^2\ga^2\iint_{Q_\de} \Theta^2\psi^2|x|^{2-\al}\df x\df t+\iint_{Q_\de} \Theta \psi^2|x|^{\al-2}\df x\df t\\
		&\leq s^2\ga^2\iint_{Q_\de} \Theta^2\psi^2|x|^{2-\al}\df x\df t+C\iint_{Q_\de} \Theta (w\nabla\psi\cdot\nabla\psi)\df x\df t, 
	\end{split}
\end{equation*}
where the constant $C>0$ depends only on $\al$ and $N$. Combining this estimate with \eqref{01.11.25}, we obtain
\begin{equation}\label{01.12.4}
	\begin{split} 
	&s\iint_{Q_\de}\Theta (w\nabla\psi\cdot\nabla\psi)\df x\df t+s^3\iint_{Q_\de} \Theta^3\psi^2|x|^{2-\al}\df x\df t\\
	&\leq C\|e^{-s\xi}f_\de\|_{L^2(Q_\de)}^2+Cs\iint_{\Ga^+\ts (0,T)} \Theta (w\nabla\psi\cdot\nabla\psi)(x\cdot\nu)\df S\df t
	\end{split} 
\end{equation}
for all $s\geq s_0$, where the constants $s_0\geq 1$ and $C>0$ depend only on $\al, T, N$, and $\ga$.

Finally, \eqref{01.12.4}, \eqref{01.06.18}, and \eqref{01.11.29} yield
\begin{equation}\label{01.12.7}
	\begin{split} 
	&s\iint_{Q_\de} \Theta (|x|^\al\nabla \vp_\de\cdot \nabla \vp_\de)e^{-2s\xi}\df x\df t+s^3\iint_{Q_\de}\Theta^3\vp_\de^2|x|^{2-\al}e^{-2s\xi}\df x\df t\\
	&\leq C\|e^{-s\xi}f_\de\|_{L^2(Q_\de)}^2+Cs\iint_{\Ga^+\ts (0,T)} \Theta (w\nabla\vp_\de\cdot\nabla\vp_\de)(x\cdot\nu)e^{-2s\xi}\df S\df t
	\end{split} 
\end{equation}
for all $s\geq s_0$, where the constants $s_0\geq 1$ and $C>0$ depend only on $\al, T, N$, and $\ga$.

\subsection{Approximation} 

Let $\vp$ be the solution of the backward equation associated with \eqref{12.14.1}:
\begin{equation}\label{01.12.8}
	\begin{cases}
		\pt_t\vp-\mcA \vp=0, &\mbox{in }Q, \\
			\vp=0, &\mbox{on }\pt Q, \\
		\vp(T)=\vp^T, &\mbox{in }\Om, 
	\end{cases}
\end{equation}
where $\vp^T\in C_0^\iy(\Om)$. Let $\vp_\de$ be the solution of \eqref{01.06.19} with respect to $(\vp^T, 0)$ for each $\de\in (0,\dist(\supp\vp^T,\pt\Om))$.

From \eqref{01.12.7}, \eqref{01.11.18}, and \eqref{01.12.5}, together with the boundedness of $\Theta e^{-s\xi}$, $\Theta^3|x|^{2-\al}e^{-2s\xi}$, and $\Theta (x\cdot\nu)e^{-2s\xi}$, we obtain
\begin{equation}\label{01.12.6}
	\begin{split} 
	&s\iint_{Q} \Theta (|x|^\al\nabla \vp\cdot \nabla \vp)e^{-2s\xi}\df x\df t+s^3\iint_{Q}\Theta^3\vp^2|x|^{2-\al}e^{-2s\xi}\df x\df t\\
	&\leq Cs\iint_{\Ga^+\ts (0,T)} \Theta (w\nabla\vp\cdot\nabla\vp)(x\cdot\nu)e^{-2s\xi}\df S\df t
	\end{split} 
\end{equation}
 for all $s\geq s_0$, where the constants $s_0\geq 1$ and $C>0$ depend only on $\al, T, N$, and $\ga$.

Finally, let $\vp^T\in H_0^1(\Om;w)$. Choose a sequence $\vp_n^T\in C_0^\iy(\Om)$ such that $\vp_n^T\ra \vp^T$ strongly in $H_0^1(\Om;w)$ as $n\ra\iy$, and let $\vp_n$ be the solution of \eqref{01.12.8} with respect to $(\vp_n^T,0)$. Applying the same estimate as \eqref{01.11.8} to the difference $\vp_n-\vp$, we obtain
\begin{equation*}
	\vp_n\ra \vp \mbox{ strongly in } L^2(0,T; D(\mcA)), 
\end{equation*}
and 
\begin{equation*}
	\f{\pt\vp_n}{\pt\nu}\ra \f{\pt\vp}{\pt\nu} \mbox{ strongly in } L^2(0,T; L^2(\Ga^+))
\end{equation*}
by the standard Sobolev trace theorem on $\Om-B(0,R_0)$.
	Estimate \eqref{01.12.6} holds for each $\vp_n$. Letting $n\ra\iy$, we infer that \eqref{01.12.6} also holds for $\vp$. This proves the following theorem.
 
\begin{theorem}\label{01.12.T1}
	Let $\vp^T\in H_0^1(\Om;w)$, and let $\vp$ be the solution of \eqref{01.12.8} with respect to $(\vp^T,0)$. Then there exist constants $s_0\geq 1$ and $C>0$, depending only on $\al$, $N$, $T$, and $\ga$, such that, for all $s\geq s_0$, we have
	\begin{equation}\label{01.12.9}
		\begin{split} 
			&s\iint_{Q} \Theta (|x|^\al\nabla \vp\cdot \nabla \vp)e^{-2s\xi}\df x\df t+s^3\iint_{Q}\Theta^3\vp^2|x|^{2-\al}e^{-2s\xi}\df x\df t\\
			&\leq Cs\iint_{\Ga^+\ts (0,T)} \Theta w\left(\f{\pt \vp}{\pt\nu}\right)^2(x\cdot\nu)e^{-2s\xi}\df S\df t.
		\end{split} 
	\end{equation}
\end{theorem}

\subsection{Boundary Observability for the Limiting Equation}

Let $s_0\geq 1$ be as in Theorem \ref{01.12.T1}. Since $1\leq \ga -\eta \leq |\eta|_\iy+1$ on $\Om$, we have
\begin{equation*}
	\Theta e^{-2s_0\xi}\geq C \mbox{ on } \Om\ts \left(\f{1}{4}T, \f{3}{4}T\right), \quad \Theta^3 e^{-2s_0\xi}\geq C \mbox{ on } \Om\ts \left(\f{1}{4}T, \f{3}{4}T\right), 
\end{equation*}
and 
\begin{equation*}
	\Theta e^{-2s_0\xi}\leq C \mbox{ on }Q, 
\end{equation*}
Hence,
\begin{equation*}
	\int_{\f{1}{4}T}^{\f{3}{4}T}\int_\Om |x|^\al \nabla\vp\cdot\nabla\vp \df x\df t+\int_{\f{1}{4}T}^{\f{3}{4}T}\int_\Om |x|^{2-\al}\vp^2\df x\df t\leq C\iint_{\Ga^+\ts (0,T)}\left(\f{\pt \vp}{\pt\nu}\right)^2\df S\df t,
\end{equation*}
where the constant $C>0$ depends only on $\al, T, N, M, \ga$, and $\sup_{x\in \pt\Om} x\cdot\nu$.
Consequently,
\begin{equation*}
	\begin{split} 
	\int_{\f{1}{4}T}^{\f{3}{4}T}\int_\Om \vp^2\df x\df t
	&\leq \int_{\f{1}{4}T}^{\f{3}{4}T}\int_\Om \vp^2|x|^{2-\al}\df x\df t+\int_{\f{1}{4}T}^{\f{3}{4}T}\int_\Om\vp^2|x|^{\al-2}\df x\df t\\
	&\leq \int_{\f{1}{4}T}^{\f{3}{4}T}\int_\Om\vp^2|x|^{2-\al}\df x\df t+C\int_{\f{1}{4}T}^{\f{3}{4}T}\int_\Om|x|^\al \nabla\vp\cdot\nabla\vp\df x\df t\\
	&\leq C\iint_{\Ga^+\ts (0,T)}\left(\f{\pt \vp}{\pt\nu}\right)^2\df S\df t
	\end{split}
		\end{equation*}
		by Lemma \ref{08.15.L1}, where the constant $C>0$ depends only on $\al, T, N, M, \ga$, and $\sup_{x\in \pt\Om} x\cdot\nu$.
		On the other hand, multiplying \eqref{01.12.8} by $\vp$, using $\vp^T\in H_0^1(\Om;w)$, and integrating over $\Om\ts(0,t)$, we obtain
	\begin{equation*}
		0=\f{1}{2}\|\vp(t)\|_{L^2(\Om)}^2-\f{1}{2}\|\vp(0)\|_{L^2(\Om)}^2-\int_0^t\int_\Om |x|^\al \nabla\vp\cdot\nabla\vp\df x\df t.
	\end{equation*}
	Thus,
\begin{equation*}
	\|\vp(0)\|_{L^2(\Om)}\leq \|\vp(t)\|_{L^2(\Om)} \mbox{ for all } t\in [0,T]. 
\end{equation*}
Therefore,
\begin{equation}\label{01.12.10}
	\begin{split} 
	\int_\Om \vp^2(0)\df x\leq \f{2}{T}\int_{\f{1}{4}T}^{\f{3}{4}T}\int_\Om \vp^2\df x\df t
	&\leq C\iint_{\Ga^+\ts (0,T)} \left(\f{\pt\vp}{\pt\nu}\right)^2\df S\df t,
	\end{split} 
\end{equation}
where the constant $C>0$ depends only on $\al, T, N, M, \ga$, and $\sup_{x\in \pt\Om} x\cdot\nu$.

Finally, let $\vp^T\in L^2(\Om)$, and write
\[
\vp^T=\sum_{n=1}^\iy(\vp^T, \Phi_n)_{L^2(\Om)}\Phi_n.
\]
Then
\[
\vp(t)=\sum_{n=1}^\iy (\vp^T,\Phi_n)_{L^2(\Om)}e^{-\la_n(T-t)}\Phi_n
\]
is the weak solution of \eqref{01.12.8} with respect to $(\vp^T, 0)$, and $\vp(t)\in H_0^1(\Om;w)$ for all $t\in [0,T)$. In particular,
\[
\vp\left(\f{1}{2}T\right)=\sum_{n=1}^\iy (\vp^T,\Phi_n)_{L^2(\Om)}e^{-\frac{1}{2}\la_n T}\Phi_n\in H_0^1(\Om;w).
\]
We may therefore apply the previous argument on the interval $\Om\ts \left(0,\f{1}{2}T\right)$ to the same solution $\vp$, and obtain
\begin{equation*}
	\begin{split}
		\int_\Om\vp^2(0)\df x
		&\leq C\iint_{\Ga^+\ts \left(0,\f{1}{2}T\right)}  \left(\f{\pt\vp}{\pt \nu}\right)^2\df S\df t\leq C\iint_{\Ga^+\ts (0,T)}  \left(\f{\pt\vp}{\pt\nu}\right)^2\df S\df t, 
	\end{split}
\end{equation*}
where the constant $C>0$ depends only on $\al, T, N, M, \ga$, and $\sup_{x\in \pt\Om} x\cdot\nu$.
This proves the following theorem.
	
\begin{theorem}\label{01.12.T2}
	Let $\al\in (0,1)$ and $\vp^T\in L^2(\Om)$. Then the following boundary observability inequality holds
	\begin{equation}\label{01.12.11}
		\int_\Om\vp^2(0)\df x
		\leq C\iint_{\Ga^+\ts (0,T)} \left(\f{\pt\vp}{\pt\nu}\right)^2\df S\df t, 
	\end{equation}
		where $\vp$ is the weak solution of \eqref{01.12.8} with respect to $(\vp^T,0)$, and the constant $C>0$ depends only on $\al, T, N, M, \ga$, and $\sup_{x\in \pt\Om} x\cdot\nu$.
\end{theorem}

\end{document}